\documentclass[1p,11pt]{elsarticle}

\usepackage{graphicx}
\usepackage{amsmath,amsxtra,amssymb,latexsym, amscd,amsthm}

\theoremstyle{plain}
\newtheorem{thm}{Theorem}[section]
\newtheorem{cor}[thm]{Corollary}

\newtheorem{lem}[thm]{Lemma}
\newtheorem{prop}[thm]{Proposition}
\theoremstyle{definition}

\theoremstyle{plain}

\theoremstyle{remark}
\newtheorem{rmk}{Remark}

\numberwithin{equation}{section}
\numberwithin{figure}{section}
\numberwithin{table}{section}

\newcommand{\M}{\operatorname{M}}
\newcommand{\Hf}{\operatorname{H}}

\begin{document}

\begin{frontmatter}

\title{\textbf{A $q$-enumeration of lozenge tilings of a hexagon with four adjacent triangles removed from the boundary}}

\author{Tri Lai\corref{cor1}}
\address{Department of Mathematics\\ University of Nebraska--Lincoln, Lincoln, NE 68588}
\cortext[cor1]{Corresponding author, email: tlai3@unl.edu, tel: 402-472-7001}

\begin{abstract}
MacMahon proved a simple product formula for the generating function of plane partitions fitting in a given box. The theorem implies a $q$-enumeration of lozenge tilings of  a semi-regular hexagon on the triangular lattice. In this paper we generalize MacMahon's classical theorem by $q$-enumerating lozenge tilings of a new family of hexagons with four adjacent triangles removed from their boundary.
\end{abstract}

\begin{keyword}
Graphical condensation \sep Lozenge tilings \sep Perfect matchings \sep Plane partitions
\MSC[2010] 05A15 \sep 05C30 \sep 05C70
\end{keyword}

\end{frontmatter}

\section{Introduction and main results}
Given $k$ positive integers $\lambda_1\geq \lambda_2\geq \dots \geq \lambda_k$, a \emph{plane partition} of shape $(\lambda_1,\lambda_2,\dots,\lambda_k)$ is an array of non-negative integers
\begin{center}
\begin{tabular}{rccccccccc}
$n_{1,1}$   &$n_{1,2}$                 &$n_{1,3}$               & $\dotsc$               &  $\dotsc$                        & $\dotsc$                            &   $n_{1,\lambda_1}$ \\\noalign{\smallskip\smallskip}
$n_{2,1}$   &  $n_{2,2}$              & $n_{2,3}$             &  $\dotsc$               & $\dotsc$                        &         $n_{2,\lambda_2}$&          \\\noalign{\smallskip\smallskip}
$\vdots$    &       $\vdots$            & $\vdots$                &        $\vdots$         &     \reflectbox{$\ddots$\quad}               &    &              \\\noalign{\smallskip\smallskip}
 $n_{k,1}$  &  $n_{k,2}$               & $n_{k,3}$              &     $\dotsc$             &   $n_{k,\lambda_k}$ &                                          &           \\\noalign{\smallskip\smallskip}
\end{tabular}
\end{center}
so that $n_{i,j}\geq n_{i,j+1}$ and $n_{i,j}\geq n_{i+1,j}$ (i.e. all rows and all columns are weakly decreasing from left to right and from top to bottom, respectively). The sum of all entries of a plane partition $\pi$ is called the \emph{volume} (or the \emph{norm}) of the plane partition, and denoted by $|\pi|$.

The plane partitions of rectangular shape $(b,b,\dots,b)$ ($a$ rows) with entries at most $c$ are usually identified with their 3-D interpretations --- piles (or stacks) of unit cubes fitting in an $a\times b\times c$ box. The latter are in bijection with \emph{lozenge tilings} of a semi-regular hexagon $Hex(a,b,c)$  of side-lengths $a,b,c,a,b,c$ (in clockwise order, starting from the northwest side\footnote{From now on, we always list the side-lengths of a hexagon on the triangular lattice in the clockwise order, starting from the northwest side.}) on the triangular lattice. Here, a \emph{lozenge} (or \emph{unit rhombus}) is union of any two unit equilateral triangles sharing an edge; and a \emph{lozenge tiling} of a region is a covering of the region by lozenges so that there are no gaps or overlaps.

Let $q$ be an indeterminate. The \emph{$q$-integer} $[n]_q$ is defined by $[n]_q:=1+q+q^2+\dotsc+q^{n-1}$. We also define the \emph{$q$-factorial} by $[n]_q!:=[1]_q\cdot[2]_q\dotsc[n]_q$, and  the \emph{$q$-hyperfactorial} function by $\Hf_q(n):=[0]_q!\cdot[1]_q!\cdot[2]_q!\dotsc[n-1]_q!$. MacMahon \cite{Mac} proved that the volume generating function of the plane partitions fitting in an $a\times b \times c$ box is given by
\begin{equation}\label{qMac}
\sum_{\pi}q^{|\pi|}=\frac{\Hf_q(a)\Hf_q(b)\Hf_q(c)\Hf_q(a+b+c)}{\Hf_q(a+b)\Hf_q(b+c)\Hf_q(c+a)},
\end{equation}
where the sum on the left-hand side is taken over all plane partitions $\pi$ fitting in an $a\times b\times c$ box.
By specializing $q=1$ in the MacMahon formula (\ref{qMac}), it follows that the number of lozenge tilings of a semi-regular hexagon $Hex(a,b,c)$ is equal to
\begin{equation}\label{Mac}
\frac{\Hf(a)\Hf(b)\Hf(c)\Hf(a+b+c)}{\Hf(a+b)\Hf(b+c)\Hf(c+a)},
\end{equation}
where $\Hf(n):=\Hf_1(n)=0!\cdot1!\cdot2!\dotsc(n-1)!$ is the ordinary hyperfactorial function.

The tiling formula (\ref{Mac}) inspired a large body of work focusing on enumeration of lozenge tilings of a hexagon with \emph{dents} or \emph{holes} (see e.g. \cite{Ciucu3}, \cite{CF14}, \cite{CEKZ}, \cite{CK}, \cite{Eisen}, \cite{Eisen2}, \cite{FK1}, \cite{FK2}, \cite{KO}, \cite{Lai16}, \cite{Ros}). Here, a \emph{dent} is a portion of the hexagon that has been removed from the boundary, while a \emph{hole} is a portion removed from inside the hexagon. %The first result in tiling enumeration of dented hexagons is due to Cohn, Larsen and Propp when they encoded the Gelfand-Tsetlin patterns as lozenge tilings of a hexagon with triangular dents on the base (see Proposition 2.1 in \cite{CLP}). Ciucu presented the first tiling enumeration of hexagons with holes by calculating the tiling number of a `punctured hexagon', a hexagon with the central unit triangle removed (see \cite{Ciucu2}). We refer the reader to e.g.\cite{Ciucu3},\cite{CF14},\cite{Ros}, \cite{CEKZ}, \cite{CK}, \cite{Eisen}, \cite{Lai16} for more enumerative results of the same flavor.

%Extending the MacMahon's classical theorem (for the case when $q=1$), Ciucu, Eisenk\"{o}lbl, Krattenthaler, and Zare \cite{CEKZ} proved a simple product formula for the number of tilings of a hexagon of side-lengths $a,b+m,c,a+m,b,c+m$ with a triangular hole of size $m$ in the center (the region was called \emph{cored hexagon} in \cite{CK}). Recently, Ciucu and Krattenthaler \cite{CK} generalized the latter result further  by extending the central triangular hole to a \emph{shamrock hole} consisting of four adjacent equilateral triangles, and showing that the tiling formula is still a simple product formula. Precisely, the \emph{shamrock} $S_{m,a,b,c}$ is the union of four equilateral triangles with sides $m$, $a$, $b$, $c$ on the triangular lattice described in  Figure \ref{shamrock2}.

\begin{figure}\centering
\includegraphics[width=6cm]{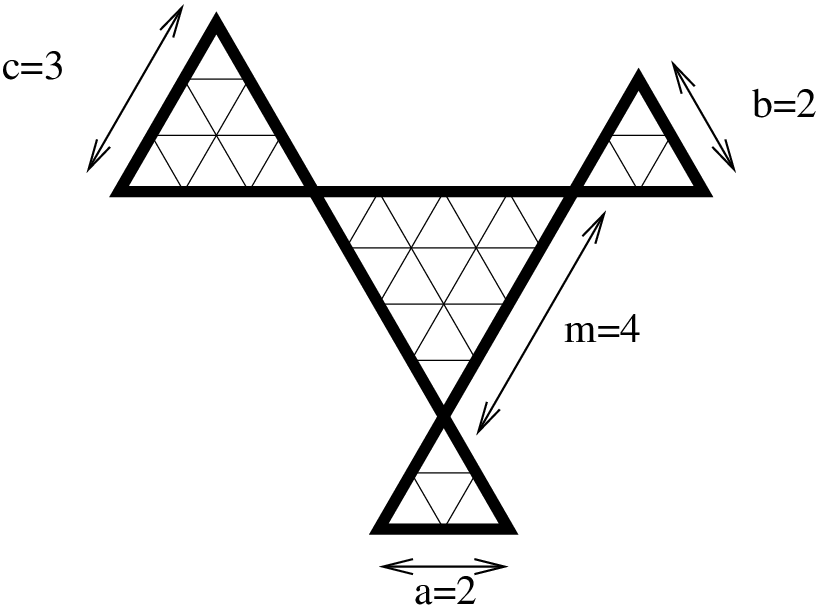}
\caption{The shamrock $S_{4,2,2,3}$.}
\label{shamrock2}
\end{figure}

In this paper, we consider a new type of dented hexagon as follows. The \emph{shamrock}\footnote{The \emph{shamrock} was first introduced by Ciucu and Krattenthaler in \cite{CK}.} $S_{m,a,b,c}$ is the union of four equilateral triangles with sides $m$, $a$, $b$, $c$ on the triangular lattice described in  Figure \ref{shamrock2}.  We start with a hexagon of side-lengths $z+a+b+c, x+y+m, t+a+b+c, z+m, x+y+a+b+c, t+m$, where $x,y,z,t$ are four non-negative integers. Next, we remove a shamrock $S_{m,a,b,c}$ from the base of the hexagon so that the lower-left vertex of the $a$-triangle in the shamrock is $x+c$ units to the right of the lower-left vertex of the hexagon. We denote by $Q\begin{pmatrix}x&y&z&t\\m&a&b&c\end{pmatrix}$ the resulting region. Figure \ref{shamrock} shows the region $Q\begin{pmatrix}4&3&2&3\\4&3&2&1\end{pmatrix}$.

\begin{figure}\centering
\includegraphics[width=8cm]{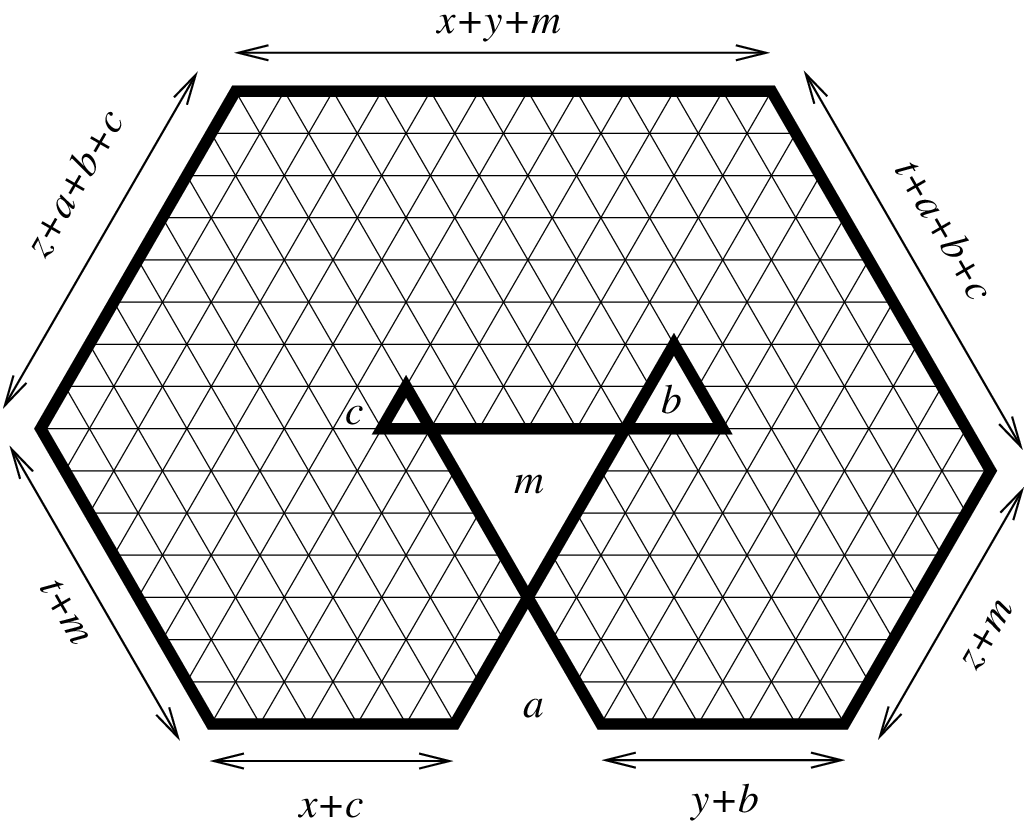}
\caption{Hexagon with a shamrock removed along the boundary.}
\label{shamrock}
\end{figure}

Our following main theorem shows that the lozenge tilings of a $Q$-type region are always enumerated by a simple product formula.
\begin{thm}\label{main}
For non-negative integers $x,y,z,t,m,a,b,c$, the number of lozenge tilings of the region $Q\begin{pmatrix}x&y&z&t\\m&a&b&c\end{pmatrix}$ is equal to
   \begin{align}\label{maineq}
 &\frac{\Hf(m+a+b+c+x+y+z+t)}{\Hf(m+a+b+c+x+y+t)\Hf(m+a+b+c+x+y+z)}\notag\\
 &\times\frac{\Hf(m+a+b+c+x+t)\Hf(m+a+b+c+x+y)\Hf(m+a+b+c+y+z)}{\Hf(m+a+b+c+z+t)\Hf(m+a+b+c+x)\Hf(m+a+b+c+y)}\notag\\
 &\times \frac{\Hf(x)\Hf(y)\Hf(z)\Hf(t)}{\Hf(x+t)\Hf(y+z)}\frac{\Hf(m)^3\Hf(a)^2\Hf(b)\Hf(c)\Hf(m+a+b+c)}{\Hf(m+a)^2\Hf(m+b)\Hf(m+c)}\notag\\
 &\times\frac{\Hf(m+b+c+z+t)\Hf(m+a+c+x)\Hf(m+a+b+y)}{\Hf(m+b+y+z)\Hf(m+c+x+t)}\notag\\
  &\times\frac{\Hf(c+x+t)\Hf(b+y+z)}{\Hf(a+c+x)\Hf(a+b+y)\Hf(b+c+z+t)}.
 \end{align}
\end{thm}
One readily sees that, by letting $a=b=c=m=0$, our $Q$-type region becomes a semi-regular hexagon. In this sense, Theorem \ref{main} is a generalization of MacMahon's tiling formula (\ref{Mac}).

\medskip

If we assign to each lozenge tiling $T$ of a semi-regular hexagon $Hex(a,b,c)$ a weight $q^{|\pi|}$, where $\pi$ is the plane partition corresponding to $T$, the generating function on the left-hand side of (\ref{qMac}) becomes a weighted sum of the lozenge tilings of $Hex(a,b,c)$. In this sense, MacMahon's formula yields a weighted enumeration of lozenge tilings of a semi-regular hexagon. This weighted enumeration is usually called a \emph{$q$-enumeration}, since the tiling weights here are all powers of $q$. While there are many known tiling enumerations of hexagons with dents or holes, most of them only concern `plain tilings' (i.e. tilings without weight). Only a few weighted enumerations have been found (see e.g. \cite{Mac}, \cite{Stanley2}, \cite[pp.374--375]{Stanley}, \cite{Lai16}). In the next part of this section, we consider such a rare  weighted enumeration, that is a $q$-analog of the result in Theorem \ref{main} (see Theorem \ref{qmain}).

\medskip

Similar to the bijection between lozenge tilings of a semi-regular hexagon $Hex(a,b,c)$ and plane partitions fitting in an $a\times b\times c$ box, one can view a lozenge tiling of $Q\begin{pmatrix}x&y&z&t\\m&a&b&c\end{pmatrix}$ as a pile of unit cubes fitting in a \emph{compound box} $\mathcal{B}:=\mathcal{B}\begin{pmatrix}x&y&z&t\\m&a&b&c\end{pmatrix}$, which is the union of 6 adjacent rectangular boxes (see Figure \ref{boxcombined}(a) for the case when $x=y=z=3, t=4, m=3, a=b=c=2$). Figure \ref{boxcombined}(b) gives a 3-D picture of the compound box $\mathcal{B}$ by showing the empty pile; the bases of the rectangular boxes in $\mathcal{B}$ consist of right-tilted lozenges and are labeled by $1,2,\dotsc,6$.

 \begin{figure}\centering
\includegraphics[width=14cm]{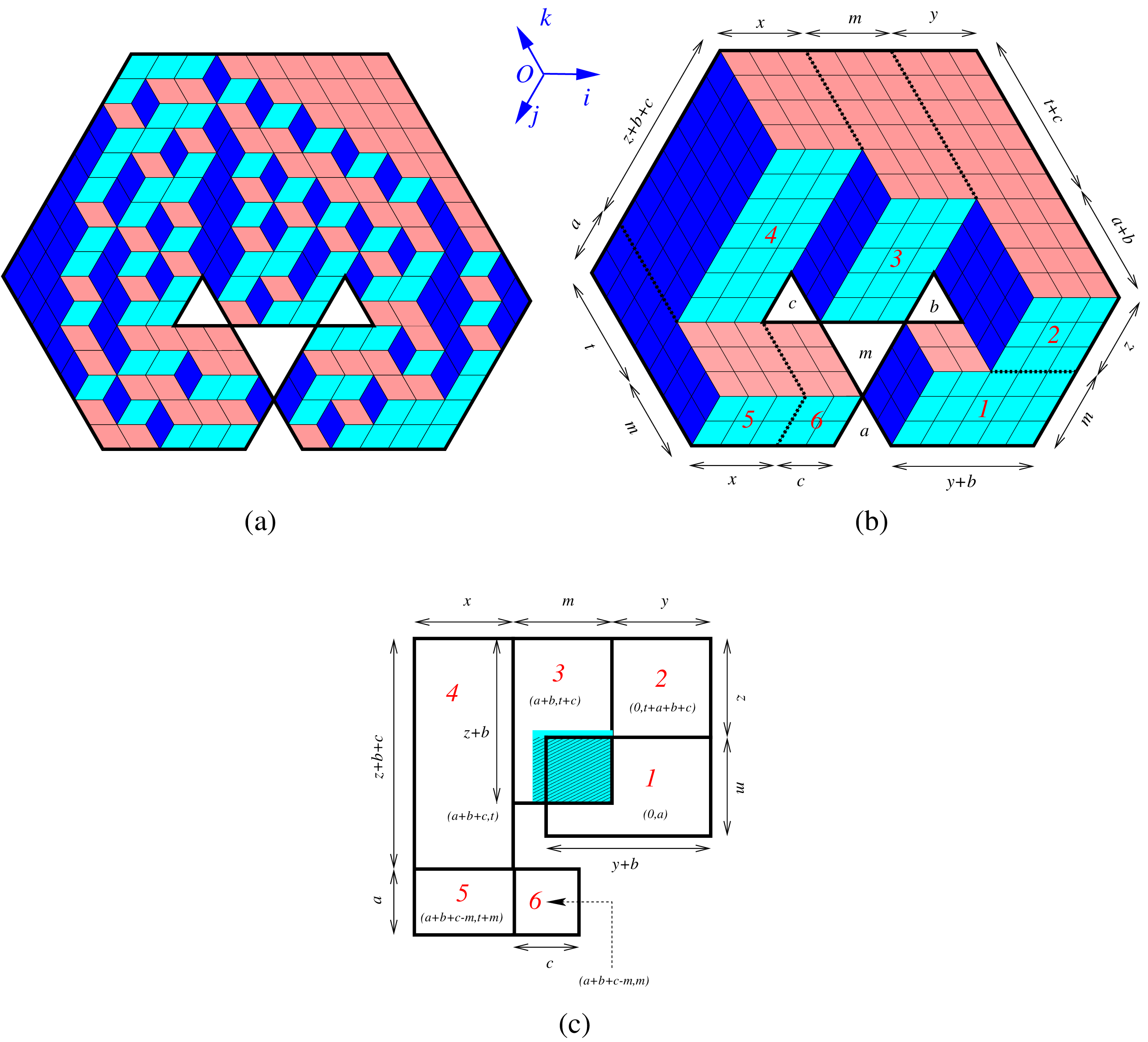}
\caption{(a) Viewing a lozenge tiling of a $Q$-type region as a pile of unit cubes fitting in a special box. (b) The tiling corresponding to the empty pile. (c) The projection of the compound box $\mathcal{B}$ on the $\textbf{Oij}$ plane. }
\label{boxcombined}
\end{figure}

% \begin{figure}\centering
%\includegraphics[width=14cm]{box5.eps}
%\caption{(a) The floor plan of the box $\mathcal{B}$. (b) Positions of 6 rooms in the $3$-D picture of the box $\mathcal{B}$. }
%\label{box5}
%\end{figure}

Projecting the compound box $\mathcal{B}$ on the $\textbf{Oij}$ plane, we get a projective diagram as in Figure \ref{boxcombined}(c). In this diagram, each rectangular box in $\mathcal{B}$ is represented by a rectangle with a pair of integers $(a,b)$, where $a$ is the level of the base and $b$ is the height of the box. We always assume that the base of the box $1$ is on level $0$. We also note that the rectangles corresponding to the boxes 1 and 3 are overlapping (the intersection is indicated by the shaded area in Figure \ref{boxcombined}(c)). However, these two boxes themselves are \emph{not} overlapping, since the box $3$ is hanging over the box 1. %We call this diagram the \emph{floor plan} of the box; and it determines our compound box. %It means that we now can define $\mathcal{B}$ to be the unique compound box, which has the floor plan given by Figure \ref{box5}(a).

We call the piles of unit cubes fitting in the compound box $\mathcal{B}$ \emph{generalized plane partitions}, since they have a similar monotonicity as (the 3D-interpretation of) the ordinary plane partitions: the tops of their columns (of unit cubes) are weakly decreasing along $\overrightarrow{\textbf{Oi}}$ and $\overrightarrow{\textbf{Oj}}$. %To precise,  two \emph{adjacent columns}  are two columns in the same room or in two adjacent rooms so that their projections on the $\textbf{Oij}$ plane are two unit squares sharing an edge. Then our monotonicity is that the top of a column does not exceed the tops of its adjacent columns on the left and behind. In the view of this, we call our stacks \emph{generalized plane partitions}. We notice that one should \emph{not} compare the heights of the columns (as in the case of ordinary plane partitions) since our columns may stay on \emph{different} levels.

Similar to MacMahon's theorem (\ref{qMac}), we have a closed form product formula for the volume generating function of the generalized plane partitions.

\begin{thm}\label{qmain} Let $m,a,b,c,x,y,z,t$ be non-negative integers. Then
   \begin{align}\label{mainequation}
 \sum_{\pi}q^{|\pi|}&= \frac{\Hf_q(m+a+b+c+x+y+z+t)}{\Hf_q(m+a+b+c+x+y+t)\Hf_q(m+a+b+c+x+y+z)}\notag\\
 &\times\frac{\Hf_q(m+a+b+c+x+t)\Hf_q(m+a+b+c+x+y)\Hf_q(m+a+b+c+y+z)}{\Hf_q(m+a+b+c+z+t)\Hf_q(m+a+b+c+x)
 \Hf_q(m+a+b+c+y)}\notag\\
 &\times \frac{\Hf_q(x)\Hf_q(y)\Hf_q(z)\Hf_q(t)}{\Hf_q(x+t)\Hf_q(y+z)}\frac{\Hf_q(m)^3\Hf_q(a)^2\Hf_q(b)\Hf_q(c)\Hf_q(m+a+b+c)}
 {\Hf_q(m+a)^2\Hf_q(m+b)\Hf_q(m+c)}\notag\\
 &\times\frac{\Hf_q(m+b+c+z+t)\Hf_q(m+a+c+x)\Hf_q(m+a+b+y)}{\Hf_q(m+b+y+z)\Hf_q(m+c+x+t)}\notag\\
  &\times\frac{\Hf_q(c+x+t)\Hf_q(b+y+z)}{\Hf_q(a+c+x)\Hf_q(a+b+y)\Hf_q(b+c+z+t)},
 \end{align}
 where the sum on the left-hand side is taken over all generalized plane partitions $\pi$ fitting in the compound box $\mathcal{B}\begin{pmatrix}x&y&z&t\\m&a&b&c\end{pmatrix}$, and where $|\pi|$ is the volume of $\pi$ (i.e. the number of unit cubes in $\pi$).
\end{thm}

Denote by $F\left(\begin{matrix}x&y&z&t\\m&a&b&c\end{matrix}\ ;\  q\right)$ the volume generating function on the left-hand side of (\ref{mainequation}). One readily sees that $F\left(\begin{matrix}x&y&z&t\\m&a&b&c\end{matrix}\ ;\  1\right)$ is exactly the number of lozenge tilings of the region $Q\begin{pmatrix}x&y&z&t\\m&a&b&c\end{pmatrix}$ and that Theorem \ref{qmain} implies Theorem \ref{main} by specializing $q=1$. 

We notice that the total volume of the compound box $\mathcal{B}\begin{pmatrix}x&y&z&t\\m&a&b&c\end{pmatrix}$, and hence the degree of the volume generating function
$F\left(\begin{matrix}x&y&z&t\\m&a&b&c\end{matrix}\ ;\  q\right)$  in $q$, is
\begin{equation}m(b + y)a + zy(t + a + b + c) + (b + z)m(c + t) + (z + b + c)xt + ax(t + m) + acm.\end{equation}

We would also like to point out that taking the complement of a generalized plane partition with respect to the compound box provides a natural involution on the set of all generalized plane partitions in this box, and that this gives the symmetry
\begin{align} q^{m(b+y)a+zy(t+a+b+c)+(b+z)m(c+t)+(z+b+c)xt+ax(t+m)+acm} &F\left(\begin{matrix}x&y&z&t\\m&a&b&c\end{matrix}\ ;\  q^{-1}\right) = \notag\\
&\qquad F\left(\begin{matrix}x&y&z&t\\m&a&b&c\end{matrix}\ ;\  q\right). \end{align}

\begin{figure}\centering
\includegraphics[width=6cm]{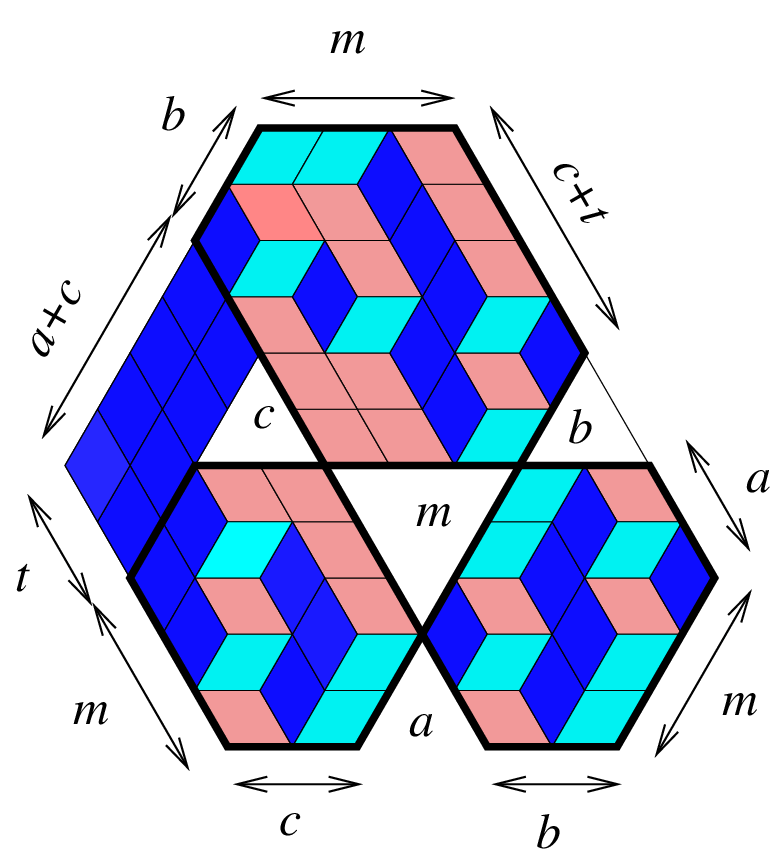}
\caption{Partitioning the compound box into three rectangular boxes in the case of $x=y=z=0$.}\label{threebox}
\end{figure}

One readily sees that MacMahon's classical theorem  (\ref{qMac}) is a special case of Theorem \ref{qmain}, when $m=a=b=c=0$, i.e.
\begin{equation}F\left(\begin{matrix}x&y&z&t\\0&0&0&0\end{matrix}\ ;\  q\right) =P_q(z,x+y,t),\end{equation}
where $P_q(a,b,c)$ denotes the MacMahon generating function in (\ref{qMac}). In addition to the above `natural' reduction, the opposite extreme case of $x = y = z = 0$ is also interesting, since it gives a reduction to a product of three MacMahon generating functions, i.e.,
\begin{equation}\label{threeboxeq}F\left(\begin{matrix}0&0&0&t\\m&a&b&c\end{matrix}\ ;\  q\right) =P_q(a,c,m)P_q(b,m,c+t)P_q(m,b,a).\end{equation}
In this case, the compound box $\mathcal{B}\begin{pmatrix}0&0&0&t\\m&a&b&c\end{pmatrix}$ is a union of three face-disjoint component boxes: the box $1$, the box $3$, and the box $6$, as indexed in Figure \ref{boxcombined} (the three other boxes are now empty). Therefore, each generalized plane partition fitting in the compound box is now a union of three independent piles of unit cubes fitting in the above three component boxes (see the three boxes corresponding the hexagons with the bold boundary in Figure \ref{threebox}). The identity (\ref{threeboxeq}) then follows.  Hence, the main result of this paper can also be regarded as a generalization of three simultaneous instances of MacMahon's theorem. The unweighted version of (\ref{threeboxeq}) (when $t=0$) also applies to the main result of Ciucu--Krattenthaler \cite{CK}, as pointed out in their Remark 1 and Figure 8.

It is worth noticing that Ciucu and Krattenthaler \cite{CK} proved a closed form product formula for the number of tilings of a hexagon with a shamrock hole in the center. However, there are \emph{not} any $q$-enumerations presented in \cite{CK}.

\medskip

The goal of the present paper is proving Theorem \ref{qmain} by using the graphical condensation method first introduced by Eric H. Kuo in \cite{Kuo04}. This condensation can be viewed as a combinatorial interpretation of Dodgson condensation (which is based on the Jacobi-Desnanot identity, see e.g. \cite{Dod} and \cite{Mui}, pp. 136--148). We refer the reader to e.g. \cite{Ciucu}, \cite{Ful}, \cite{Kuo06}, \cite{speyer}, \cite{YYZ}, \cite{YZ} for various aspects and generalizations of the method; and e.g. \cite{CF14}, \cite{CF15}, \cite{CK}, \cite{CL}, \cite{KW}, \cite{Lai15a}, \cite{Lai15b}, \cite{Lai15c}, \cite{Lai16}, \cite{Lai17}, \cite{LM}, \cite{LMNT}, \cite{Ranjan1}, \cite{Ranjan2}, \cite{Zhang}  for recent applications of Kuo condensation.

The rest of our paper is organized as follows. For ease of reference, we quote several preliminary results in Section 2, including the particular version of Kuo condensation employed in our proofs. In order to apply Kuo condensation to our $Q$-type regions, we consider several simple weight assignments on the lozenges of the regions in Section 3.   Section 4 is devoted to two generalizations of related work of Ciucu and Krattenthaler in \cite[Theorem 3.1]{CK} about the \emph{magnet bar region} (the $b=c=0$ specialization of a $Q$-type region).  Finally, we prove Theorem \ref{qmain} in Section 5.

\section{Preliminaries}

Let $G$ be a finite simple graph without loops. A \emph{perfect matching} of $G$ is a collection of edges covering each vertex of $G$ exactly once. Let $R$ be a \emph{region}\footnote{From now on, we use the word \emph{region} to mean a finite connected region on the triangular lattice.}. The \emph{(planar) dual graph} of $R$ is the graph whose vertices are unit triangles in $R$ and whose edges connect precisely two unit triangles sharing an edge. One can identify the lozenge tilings of $R$ with the perfect matchings of its dual graph.

For a weighted graph $G$, we define the \emph{matching generating function} $\M(G)$ of $G$ to be the sum of the weights of all perfect matchings in $G$, where the \emph{weight} of a perfect matching is the product of weights of its edges. If the lozenges of a region $R$ are weighted, we define similarly the \emph{tiling generating function} $\M(R)$ of $R$. In the weighted case, each edge of the dual graph $G$ of the region $R$ carries the same weight as its corresponding lozenge in $R$.

The following condensation theorem by  Kuo is the key for our proofs.
\begin{thm}[Theorem 5.1 in \cite{Kuo04}]\label{kuothm}
Let $G=(V_1,V_2,E)$ be a (weighted) bipartite planar graph in which $|V_1|=|V_2|$. Assume that  $u, v, w, s$ are four vertices appearing in a cyclic order on a face of $G$, such that $u,w \in V_1$ and $v,s \in V_2$. Then
\begin{equation}\label{kuoeq}
\M(G)\M(G-\{u, v, w, s\})=\M(G-\{u, v\})\M(G-\{ w, s\})+\M(G-\{u, s\})\M(G-\{v, w\}).
\end{equation}
\end{thm}

A \emph{forced lozenge} of a region $R$ is a lozenge that is contained in every tiling of $R$. Assume that we removed several forced lozenges $\ell_1,\ell_2,\dotsc,\ell_n$ from the region $R$, and denote by $R'$ the resulting region. Then one clearly has
\begin{equation}\label{forceeq}
\M(R)=\M(R')\prod_{i=1}^{n}wt(\ell_i),
\end{equation}
where $wt(\ell_i)$ denotes the weight of the lozenge $\ell_i$.

If a region $R$ admits a lozenge tiling, then the number of up-pointing unit triangles equals the number of down-pointing unit triangles in $R$. If a region satisfies the latter balancing condition, we say that the region is \emph{balanced}. The following lemma can be considered as a generalization of the identity (\ref{forceeq}).

\begin{lem}[Region-splitting Lemma]\label{GS}
Let $R$ be a balanced region. Assume that a subregion $S$ of $R$ satisfies the following two conditions:
\begin{enumerate}
\item[(i)] \text{\rm{(Separating Condition)}} There is only one type of unit triangle (up-pointing or down-pointing) running along each side of the border between $S$ and $R-S$.

\item[(ii)] \text{\rm{(Balancing Condition)}} $S$ is balanced.
\end{enumerate}
Then
\begin{equation}\label{GSeq}
\M(R)=\M(S)\, \M(R-S).
\end{equation}
\end{lem}
\begin{proof}
Assume there is a tiling of $R$ which contains boundary-crossing lozenges between $S$ and $R - S$ (i.e., lozenges which consist of a unit triangle from the boundary of $S$ and a unit triangle from the boundary of $R - S$). Since there is only one type of unit triangle on each side of the boundary between $S$ and $R - S$, and since $S$ and $R-S$ are balanced, the regions obtained by removing such boundary-crossing lozenges would no longer be balanced, and hence would have no tilings. Therefore, there can not be any boundary-crossing lozenges, and $S$ and $R - S$ must be tiled independently, giving the factorization (\ref{GSeq}).
\end{proof}

\section{$q$-weight assignments}

Lozenges in a region $R$ come with three different orientations: left, right, and vertical lozenges (see Figure \ref{rhumbustype}). Next, we consider three simple $q$-weight assignments of lozenges in our region $Q:=Q\begin{pmatrix}x&y&z&t\\m&a&b&c\end{pmatrix}$ as follows.

View a lozenge tiling $T$ of the region $Q$ as a generalized plane partition (i.e. a pile of unit cubes); each right lozenge is viewed as the top of a column of unit cubes. We assign to each right lozenge the weight $q^x$, where $x$ is the number of unit cubes in the corresponding column. All left and vertical lozenges are weighted by $1$. We call this assignment the \emph{natural q-weight assignment} of $Q$, and use the notation $wt_0$ for the assignment (see Figure \ref{weight1}(a) for the case when $x=y=z=3$, $t=4$, $m=3$, $a=b=c=2$).

\begin{figure}\centering
\includegraphics[width=7cm]{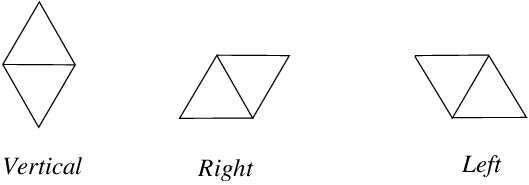}
\caption{Three orientations of lozenges.}
\label{rhumbustype}
\end{figure}

\begin{figure}\centering
\includegraphics[width=12cm]{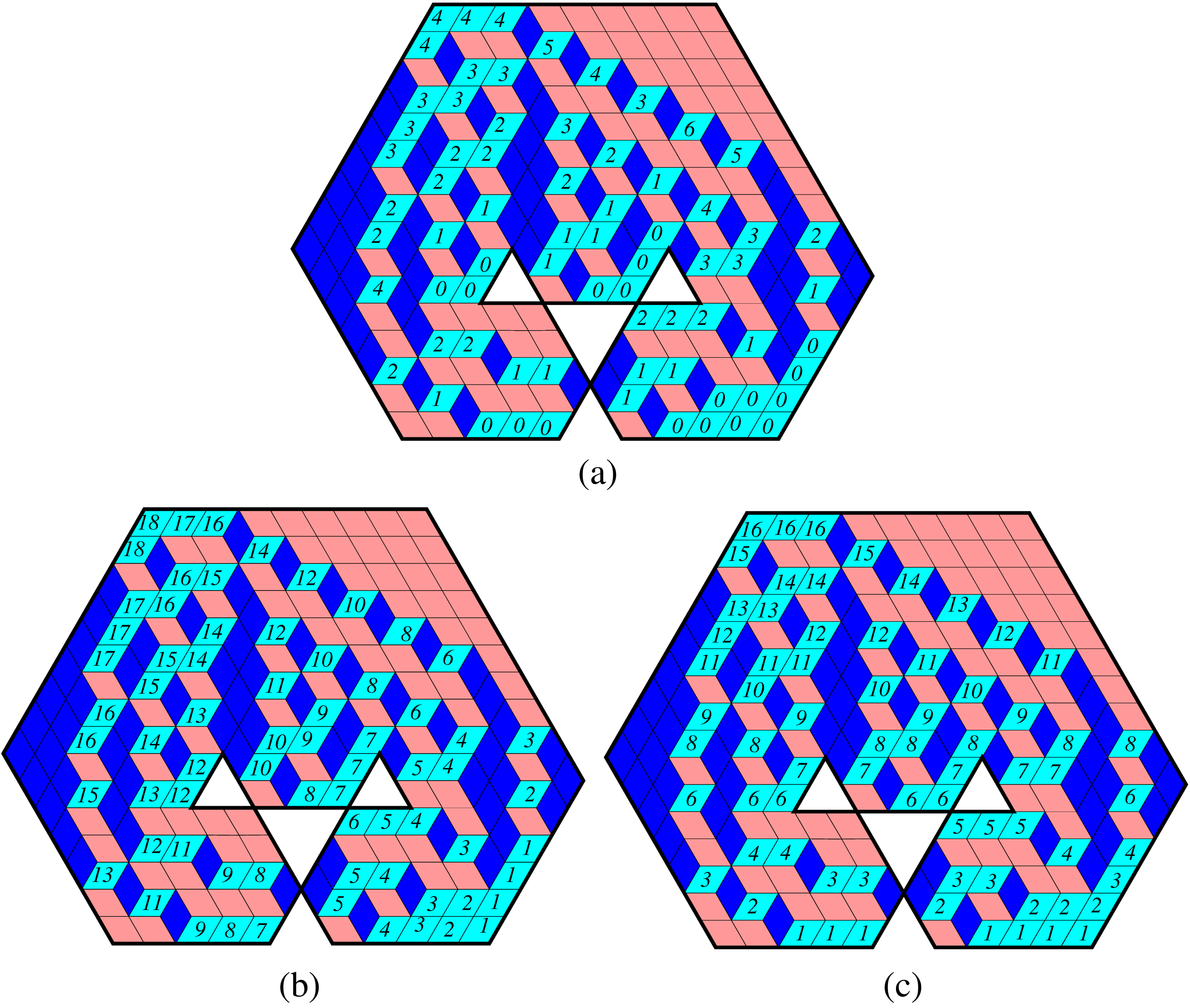}
\caption{Three $q$-weight assignments on a sample tiling of region $Q$: (a) $wt_0$, (b) $wt_1$, (c) $wt_2$. The right lozenges with label $x$ are weighted by $q^x$.}
\label{weight1}
\end{figure}

Besides the natural $q$-weight assignment $wt_0$, we consider the following two $q$-weight assignments:
\begin{enumerate}
\item[(1)] \emph{Assignment 1.}  The weights of left and vertical lozenges  are all 1. The weight of a right lozenge is $q^l$, where $l$
 is the distance between the left side of the lozenge and the southeast side of the region $Q$. We use notation $wt_1$ for this weight assignment (see Figure \ref{weight1}(b)).
\item[(2)] \emph{Assignment 2.} All left and vertical lozenges are also weighted by 1. However, a right lozenge is now weighted by $q^{n}$, where $n$ is the distance
 between the top of the lozenge and the base of the region $Q$. This assignment is  denoted by $wt_2$ (see Figure \ref{weight1}(c)).
\end{enumerate}

Let $T$ be a tiling of $Q$.  We denote by $wt_0(T)$, $wt_1(T)$ and $wt_2(T)$ the weights of the tiling $T$ with respect to the weight assignments $wt_0$, $wt_1$ and $wt_2$. We also denote by $\M_0(Q)$, $\M_1(Q)$ and $\M_2(Q)$ the tiling generating functions of $Q$ corresponding to the weight assignments $wt_0$, $wt_1$ and $wt_2$. It is easy to see that $\M_0(Q)$ is exactly the volume generation function of the generalized plane partitions corresponding to the lozenge tilings of $Q$.

In the next proposition, we will show that the three weight assignments are the same up to some multiplicative factors.

% \begin{figure}\centering
%\includegraphics[width=8cm]{3weights.eps}
%\caption{Comparing the two weights assignment of tilings of a the hexagon. The lozenge with label $x$ has weight $q^x$.}
%\label{3weights}
%\end{figure}

We define two functions
 \begin{align}
 \textbf{f}&\begin{pmatrix}x&y&z&t\\m&a&b&c\end{pmatrix}:=m\binom{y+b+1}{2}+z\binom{y+1}{2}+m(z+b)(y+a+b)+(z+b)\binom{m+1}{2}\notag\\
 &+x(z+b+c)(y+m+a+b+c)+(z+b+c)\binom{x+1}{2}+a(x+c)(y+a+b)+a\binom{x+c+1}{2}
 \end{align}
 and
  \begin{align}
\textbf{g}&\begin{pmatrix}x&y&z&t\\m&a&b&c\end{pmatrix}:=(y+b)\binom{m+1}{2}+myz+y\binom{z+1}{2}+m(z+b)(m+a)\notag\\
 &+m\binom{z+b+1}{2}+x(m+a)(z+b+c)+x\binom{z+b+c+1}{2}+(x+c)\binom{a+1}{2}.
 \end{align}
Note that $\textbf{f}\begin{pmatrix}x&y&z&t\\m&a&b&c\end{pmatrix}$ and $\textbf{g}\begin{pmatrix}x&y&z&t\\m&a&b&c\end{pmatrix}$ are both independent of $t$.

\begin{prop}\label{hexprop2} For any non-negative integers $m,a,b,c,x,y,z,t$
\begin{equation}\label{ratioweight1}
\M_1\left(Q\begin{pmatrix}x&y&z&t\\m&a&b&c\end{pmatrix}\right)
=q^{\textbf{f}\begin{pmatrix}x&y&z&t\\m&a&b&c\end{pmatrix}}\sum_{\pi}q^{|\pi|}
\end{equation}
and
\begin{equation}\label{ratioweight2}
\M_2\left(Q\begin{pmatrix}x&y&z&t\\m&a&b&c\end{pmatrix}\right)=q^{\textbf{g}\begin{pmatrix}x&y&z&t\\m&a&b&c\end{pmatrix}}\sum_{\pi}q^{|\pi|},
\end{equation}
where the sums on the right-hand sides are taken over all generalized plane partitions $\pi$ fitting in the compound box $\mathcal{B}\begin{pmatrix}x&y&z&t\\m&a&b&c\end{pmatrix}$.
\end{prop}

 \begin{figure}\centering
\includegraphics[width=12cm]{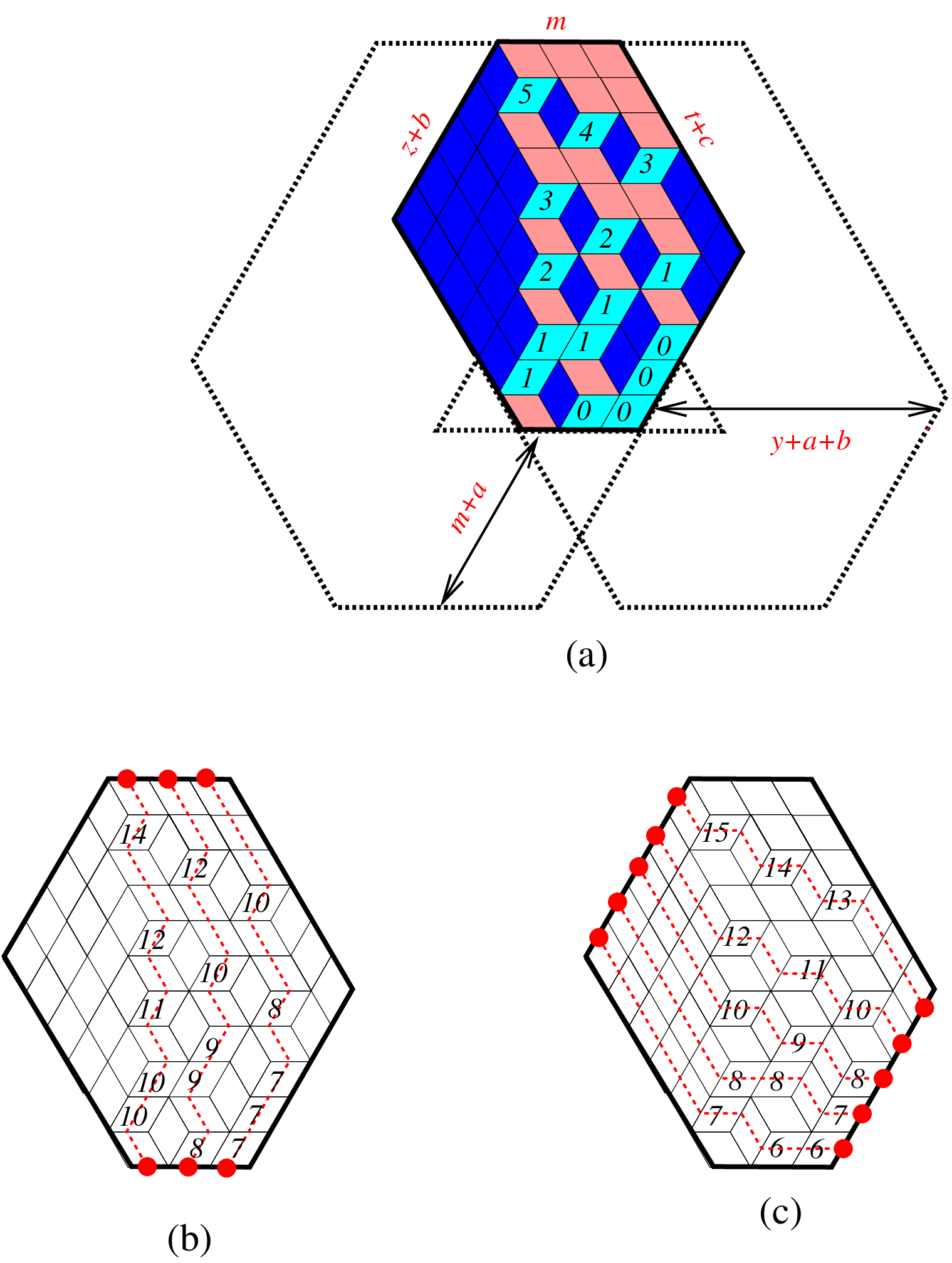}
\caption{The partial-partition corresponding to the box $3$.}
\label{box6}
\end{figure}

\begin{proof}
We use the following shorthand notations  in this proof: $\textbf{f}:=\textbf{f}\begin{pmatrix}x&y&z&t\\m&a&b&c\end{pmatrix}$,  $\textbf{g}:=\textbf{g}\begin{pmatrix}x&y&z&t\\m&a&b&c\end{pmatrix}$, $\mathcal{B}:=\mathcal{B}\begin{pmatrix}x&y&z&t\\m&a&b&c\end{pmatrix}$, and $Q:=Q\begin{pmatrix}x&y&z&t\\m&a&b&c\end{pmatrix}$.

Let $T$ be any lozenge tiling of the region $Q$, and $\pi$ the generalized plane partition corresponding to $T$. We only need to show that
\begin{equation}
\frac{wt_1(T)}{q^{|\pi|}}=q^{\textbf{f}} \quad\text{ and }\quad \frac{wt_2(T)}{q^{|\pi|}}=q^{\textbf{g}}.
\end{equation}

Assume that the box $i$ of the compound box $\mathcal{B}$ has size $a_i\times b_i \times c_i$ (for $1\leq i\leq 6$). The base of the box $i$ is depicted as a parallelogram $P_i$ consisting of right lozenges in Figure \ref{boxcombined}(b). We assume in addition that the left side of $P_i$ is $x_i$ units to the left of the southeast side of the region $Q$, and the bottom of $P_i$ is $y_i$ units above the bottom of region $Q$.

Divide the  generalized plane partition $\pi$ into $6$ disjoint partial-partitions $\pi_i$ ($1\leq i\leq 6$) fitting in the box $i$. Recall that the box $i$ is associated with a semi-regular hexagon with side lengths $a_i$, $b_i$, $c_i$, $a_i$, $b_i$, $c_i$, in clockwise order from the northwest side. Each partial-partition $\pi_i$ in turn gives a lozenge tiling $T_i$ of  the semi-regular hexagon $Hex(a_i,b_i,c_i)$. Figure \ref{box6}(a) shows the partial-partition $\pi_3$ of the generalized plane partition $\pi$ in Figure \ref{boxcombined}(a), as well as  the relative positions of the parallelogram $P_3$ to the bottom and the southeast side of the region $Q$ (for $x=3,y=3,z=3,t=4,m=3,a=2,b=2,c=2$). In particular, this gives $a_3 = z+b$, $b_3 = m$, $c_3 = t+c$, $x_3 = y+a+b$ and $y_3 =m+a$.

Apply the weight assignment $wt_1$ to the whole  tiling $T$ of the region $Q$. This yields a local weight assignment $wt_1^{(i)}$ for the tiling $T_i$ of hexagon $Hex(a_i,b_i,c_i)$. Precisely, each right lozenge in $T_i$ is  now weighted by $q^{x_i+l}$, where $l$ is the distance between the left side of the lozenge and the southeast side of the hexagon $Hex(a_i,b_i,c_i)$. Encode the tiling $T_i$ as a family of $b_i$ disjoint lozenge-paths connecting the top and the bottom of the hexagon (see the dotted paths in Figure \ref{box6}(b)). Dividing the weight of each right lozenge in the lozenge-path $j$ (from right to left) by $q^{x_i+j}$, we get the weight assignment $wt_0$ for $T_i$. Since the end points of the above lozenge-paths are fixed, each lozenge-path has exactly $a_i$ right lozenges. Thus, we have
\[\frac{wt_1^{(i)}(T_i)}{wt_0(T_i)}=\frac{wt_1^{(i)}(T_i)}{q^{|\pi_i|}}=q^{a_ib_ix_i+a_ib_i(b_i+1)/2}.\]
Multiplying all above equations for $i=1,2,\dotsc,6$, we get
\begin{equation}
\frac{wt_1(T)}{q^{|\pi|}}=q^{\sum_{i=1}^{6}(a_ib_ix_i+a_ib_i(b_i+1)/2)}.
\end{equation}
Next, we assume that the whole tiling $T$ of $Q$ is weighted by $wt_2$. We now encode the tiling $T_i$ of $Hex(a_i,b_i,c_i)$ as an $a_i$-tuple of disjoint lozenge-paths connecting the northwest side and the southeast side of the hexagon (see Figure \ref{box6}(c)).  Dividing each right lozenge in the lozenge-path $j$ (from bottom to top) by $q^{y_i+j}$, we get back again the weight assignment $wt_0$. We note that each lozenge-path has now $b_i$ right lozenges. Similar to the case of $wt_1$, we have
\begin{equation}
\frac{wt_2(T)}{q^{|\pi|}}=q^{\sum_{i=1}^{6}(a_ib_iy_i+b_ia_i(a_i+1)/2)}.
\end{equation}
Obtaining the formulas for $a_i,b_i,x_i,y_i$ in terms of $m,a,b,c,x,y,z,t$ from Figures \ref{boxcombined}(b) and (c), we get $\textbf{f}=\sum_{i=1}^{6}(a_ib_ix_i+a_ib_i(b_i+1)/2)$ and $\textbf{g}=\sum_{i=1}^{6}(a_ib_iy_i+b_ia_i(a_i+1)/2)$. This finishes our proof.
\end{proof}
We note that the powers $q^{\textbf{f}}$ and $q^{\textbf{g}}$ in the above proposition are exactly the weights $wt_1(T_0)$ and $wt_2(T_0)$ of the tiling $T_0$ corresponding to the empty pile in Figure \ref{boxcombined}(b).

View the hexagon $Hex(a,b,c)$ as a special case of the region  $Q:=Q\begin{pmatrix}x&y&z&t\\m&a&b&c\end{pmatrix}$ with an empty shamrock hole, i.e., set $m = a = b = c = 0$, and then replace $z$, $x + y$ and $t$ by $a,$ $b$ and $c$, respectively. We then have the following consequence of  Proposition \ref{hexprop2} and MacMahon's $q$-formula (\ref{qMac}).

\begin{cor}\label{hexprop1} For non-negative integers $a,b,c$
\begin{equation} \label{hexeq1}
\M_1\big(Hex(a,b,c)\big)=q^{ab(b+1)/2}\frac{\Hf_q(a)\Hf_q(b)\Hf_q(c)\Hf_q(a+b+c)}{\Hf_q(a+b)\Hf_q(b+c)\Hf_q(c+a)}
\end{equation}
and
\begin{equation}\label{hexeq2}
\M_2\big(Hex(a,b,c)\big)=q^{ba(a+1)/2}\frac{\Hf_q(a)\Hf_q(b)\Hf_q(c)\Hf_q(a+b+c)}{\Hf_q(a+b)\Hf_q(b+c)\Hf_q(c+a)}.
\end{equation}
\end{cor}

The following definitions will be used in the proof of the next lemma.
A \emph{column-strict plane partition} is a plane partition having columns strictly decreasing. A \textit{semihexagon} $SH_{a,b}$ is the upper half of a lozenge hexagon $Hex(a,b,a)$.
We are interested in the lozenge tilings of the semihexagon $SH_{a,b}$, where $a$ up-pointing unit triangles at the positions $s_1,s_2,\dots,s_a$  have been removed from the base. Denote by $SH_{a,b}(s_1,s_2,\dots,s_a)$ the resulting semihexagon with dents (see Figure \ref{holeeyhex}(a) for the region $SH_{7,5}(1,2,6,7,10,11,12)$).  Assume that the lozenges in the semihexagon are weighted by $wt_2$, and we still use the notation $\M_2$ for the corresponding tiling generating function of the semihexagon with dents. There is a well-known (weight preserving) bijection between the lozenge tilings of $SH_{a,b}(s_1,s_2,\dots,s_a)$ and the column-strict plane partitions of shape $(s_{a}-a,s_{a-1}-a+1,\dotsc,s_1-1)$ with positive entries at most $a$, i.e. $wt_2(T)=q^{|\pi_T|}$, where $\pi_T$ is the plane partition corresponding to the tiling $T$ (see e.g. \cite{CLP} and \cite{Car}).

 \begin{figure}\centering
\includegraphics[width=14cm]{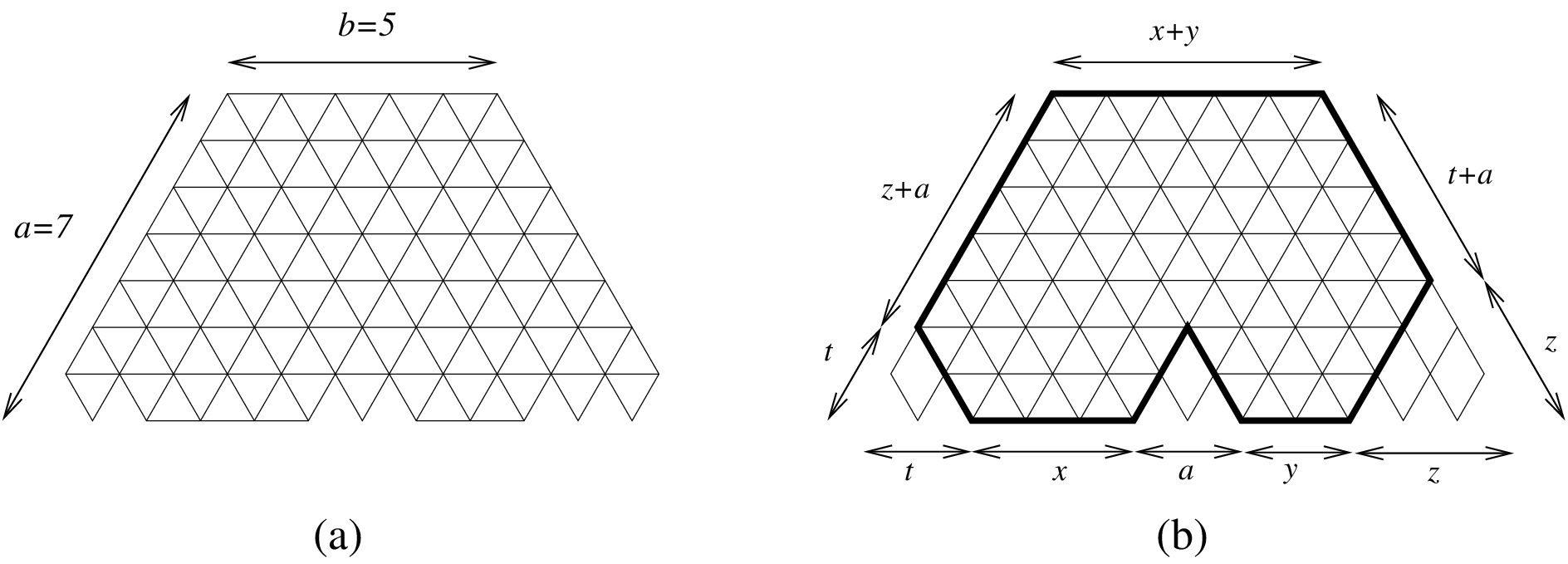}
\caption{(a) The semihexagon with dents $SH_{7,5}(1,2,6,7,10,11,12)$. (b) Obtaining the region $K_{a}(x,y,z,t)$ (restricted by the bold contour) from a semihexagon with dents by removing forced lozenges.}
\label{holeeyhex}
\end{figure}

We have the following $q$-enumeration of the lozenge tilings of a hexagon with a triangular hole on the base $K_{a}(x,y,z,t)$ (defined as the region restricted by the bold contour in Figure \ref{holeeyhex}(b)).
Note that $K_{a}(x, y, z, t)$ is the region $Q\begin{pmatrix}x&y&z&t\\0&a&0&0\end{pmatrix}$.
 \begin{lem}\label{hexprop3b} For non-negative $a,x,y,z,t$
   \begin{align}
 \M_2\big(K_{a}(x,y,z,t)\big)
&=q^{y\binom{z+1}{2}+x\binom{a+z+1}{2}}\frac{\Hf_q(a)\Hf_q(x)\Hf_q(y)\Hf_q(z)\Hf_q(t)}{\Hf_q(x+t)\Hf_q(a+x)\Hf_q(a+y)\Hf_q(y+z)}\notag\\
&\times \frac{\Hf_q(a+x+t)\Hf_q(a+x+y)\Hf_q(a+y+z)\Hf_q(a+x+y+z+t)}{\Hf_q(a+x+y+t)\Hf_q(a+x+y+z)\Hf_q(a+t+z)}.
 \end{align}
 \end{lem}
 \begin{proof}
By the above bijection between lozenge tilings of the semihexagon and the column-strict plane partitions, we have
\begin{equation}\label{semieq}
 \M_2\big(SH_{a,b}(s_1,s_2,\dotsc,s_a)\big)=\sum_{\mu}q^{|\mu|}= q^{\sum_{i=1}^{a}(s_i-i)}\prod_{1\leq i<j\leq a}\frac{q^{s_j}-q^{s_i}}{q^{j}-q^{i}},\end{equation}
 where the sum after the first equality sign is taken over all  column-strict plane partitions $\mu$ of shape  $(s_{a}-a,s_{a-1}-a+1,\dotsc,s_1-1)$ with positive entries at most $a$. For the second equality see e.g. \cite[pp. 374--375]{Stanley}.

The region  $K_{a}(x,y,z,t)$ is obtained by removing forced vertical lozenges from the semihexagon $SH_{a+z+t,x+y}$ with dents at the positions
   $\{1,2,..,t\} \cup\{t+x+1,t+x+2,\dotsc,t+x+a\} \cup \{t+x+a+y+1,t+x+a+y+2,\dotsc,t+x+a+y+z\}$.  Thus, our lemma follows from (\ref{semieq}).
 \end{proof}

\section{Two $q$-enumerations of magnet bar regions}

\begin{figure}\centering
\includegraphics[width=7cm]{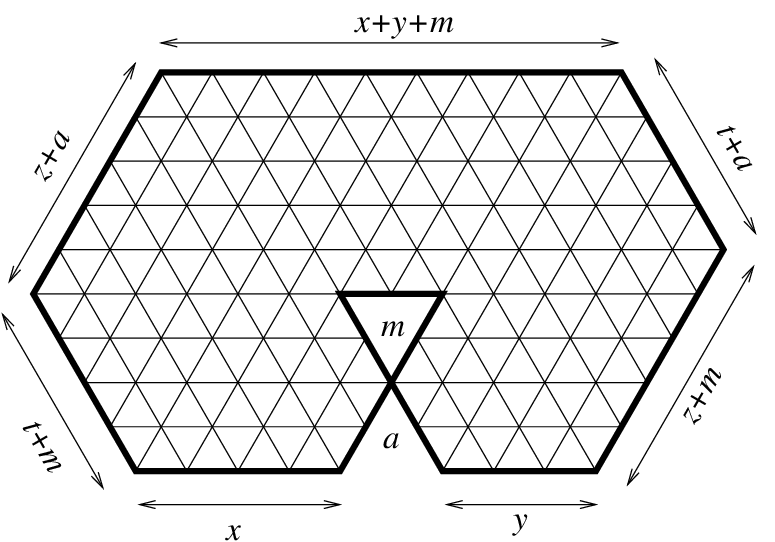}
\caption{The magnet bar $B_{2,2}(4,3,3,2)$.}
\label{magnetbar}
\end{figure}

When  $b=c=0$, our region $Q\begin{pmatrix}x&y&z&t\\m&a&b&c\end{pmatrix}$ becomes a \emph{magnet bar region} $B_{m,a}(x,y,z,t)$ first introduced  by Ciucu and Krattenthaler in \cite{CK}.   Figure \ref{magnetbar} shows the magnet bar region $B_{2,2}(4,3,3,2)$.   Ciucu and Krattenthaler \cite{CK} proved a simple product formula for the tiling number of a magnet bar region.
%
%\begin{thm}[Theorem 3.1 in \cite{CK}]\label{magnetthm}  For non-negative integers $x,y,z,t,m,a$
%\begin{align}\
% \M\big(B_{m,a}(x,y,z,t)\big)
%&= \frac{\Hf(m+a+x+y+z+t)}{\Hf(m+a+x+y+t)\Hf(m+a+x+y+z)}\notag\\
% &\times\frac{\Hf(m+a+x+t)\Hf(m+a+x+y)\Hf(m+a+y+z)}{\Hf(m+a+z+t)\Hf(m+a+x)\Hf(m+a+y)}\notag\\
% &\times \frac{\Hf(x)\Hf(y)\Hf(z)\Hf(t)\Hf(m)\Hf(a)^2}{\Hf(a+x)\Hf(a+y)\Hf(z+t)\Hf(m+a)}\notag\\
% &\times\frac{\Hf(m+z+t)\Hf(m+a+x)\Hf(m+a+y)}{\Hf(m+y+z)\Hf(m+x+t)}.
%\end{align}
%\end{thm}
In this section, we generalize their result by $q$-enumerating lozenge tilings of the magnet bar region $B_{m,a}(x,y,z,t)$. Our $q$-enumerations will be used in the proof of Theorem \ref{qmain}.

\begin{prop}\label{hexprop4} For non-negative integers $m,a,x,y,z,t$
\begin{align}\label{qmagnet}
 \M_2\big(B_{m,a}(x,y,z,t)\big)&=q^{y\binom{m+1}{2}+(m+x+y)\binom{z+1}{2}+myz+(m+a)(x+m)z+x\binom{a+1}{2}} \notag\\
&\times\frac{\Hf_q(m+a+x+y+z+t)}{\Hf_q(m+a+x+y+t)\Hf_q(m+a+x+y+z)}\notag\\
 &\times\frac{\Hf_q(m+a+x+t)\Hf_q(m+a+x+y)\Hf_q(m+a+y+z)}{\Hf_q(m+a+z+t)\Hf_q(m+a+x)
 \Hf_q(m+a+y)}\notag\\
 &\times \frac{\Hf_q(x)\Hf_q(y)\Hf_q(z)\Hf_q(t)\Hf_q(m)\Hf_q(a)^2}{\Hf_q(a+x)\Hf_q(a+y)\Hf_q(z+t)\Hf_q(m+a)}\notag\\
 &\times\frac{\Hf_q(m+z+t)\Hf_q(m+a+x)\Hf_q(m+a+y)}{\Hf_q(m+y+z)\Hf_q(m+x+t)}.
 \end{align}
\end{prop}
\begin{proof}
We prove (\ref{qmagnet}) by induction on $y+z+t$.  Throughout this proof, we assume that our magnet bar region is weighted by $wt_2$.

Our main goal is to obtain a recurrence for the $\M_2$-generating function of the magnet bar by using Kuo's condensation Theorem \ref{kuothm}. In order to apply the theorem, certain side-lengths of the magnet bar must be large enough, which requires $x+y+m\geq 2$ and $t+a\geq2$. Taking into account also the base cases of the recurrence, we need to consider the situations when $m=0$, $a=0$, $y=0$, $z=0$ or $t=0$.

If $m=0$, then our magnet bar region becomes the region $K_{a}(x,y,z,t)$ in Lemma \ref{hexprop3b}, and (\ref{qmagnet}) follows.

If $a=0$, by removing forced lozenges along the base of the region $B_{m,0}(x,y,z,t)$, we get the weighted hexagon $Hex(z,x+y+m,t)$ in which a right lozenge is weighted by $q^{m+l}$, where $l$ is the distance from the top of the lozenge to the bottom of the hexagon (see Figure \ref{basecase2}(e)). By dividing the weight of each right lozenge of the hexagon by $q^{m}$, we get back the weight assignment $wt_2$.  Since the product of weights of the forced lozenges in Figure \ref{basecase2}(e)  is  equal to $q^{y\binom{m+1}{2}}$, we get from (\ref{forceeq})
\begin{equation}
\M_2\big(B_{m,0}(x,y,z,t)\big)=q^{y\binom{m+1}{2}}q^{mz(x+y+m)}\M_2\big(Hex(z,x+y+m,t)\big),
\end{equation}
where the factor $q^{mz(x+y+m)}$ comes from the weight division. Then  (\ref{qmagnet}) follows  from Corollary \ref{hexprop1}.

%If $x=0$, by removing vertical forced lozenges along the southwest side of $B_{m,a}(0,y,z,t)$, we get a weighted hexagon with a triangular hole on its southwest side. We rotate the later region $60^0$ counter-clockwise and reflect the resulting region over a vertical line. We get a weighted region $K_{m}(a,t,z,y)$, where lozenges are weighted by the weight assignment $wt_1$. Thus,
%\begin{equation}
%\M_2\left(B_{m,a}(0,y,z,t)\right)=\M_1(K_{m}(a,t,z,y))),
%\end{equation}
%and (\ref{qmagnet}) follows from Proposition \ref{hexprop2} and Corollary \ref{hexprop3b} (note that $K_{m}(x,y,z,t)$ is exactly the region $Q\begin{pmatrix}x&y&z&t\\0&a&0&0\end{pmatrix}$). The case $y=0$ can be treated similarly in Figure \ref{basecase2}(b).

If $y=0$, after removing forced vertical lozenges (which have the weight 1), we get a new weighted region $R$ (the region restricted by the bold contour in Figure \ref{basecase2}(b)). By rotating $R$ $60^{\circ}$ clockwise and reflecting the resulting region about a vertical line, we get the region $K_{m}(z,a,x,t)$ weighted by $wt_1$. Thus, we have
\begin{equation}
\M_2\big(B_{m,a}(x,0,z,t)\big)=\M_1\big(K_{m}(z,a,x,t)\big),
\end{equation}
and (\ref{qmagnet}) follows from Proposition \ref{hexprop2} and Lemma \ref{hexprop3b}.

 \begin{figure}\centering
\includegraphics[width=15cm]{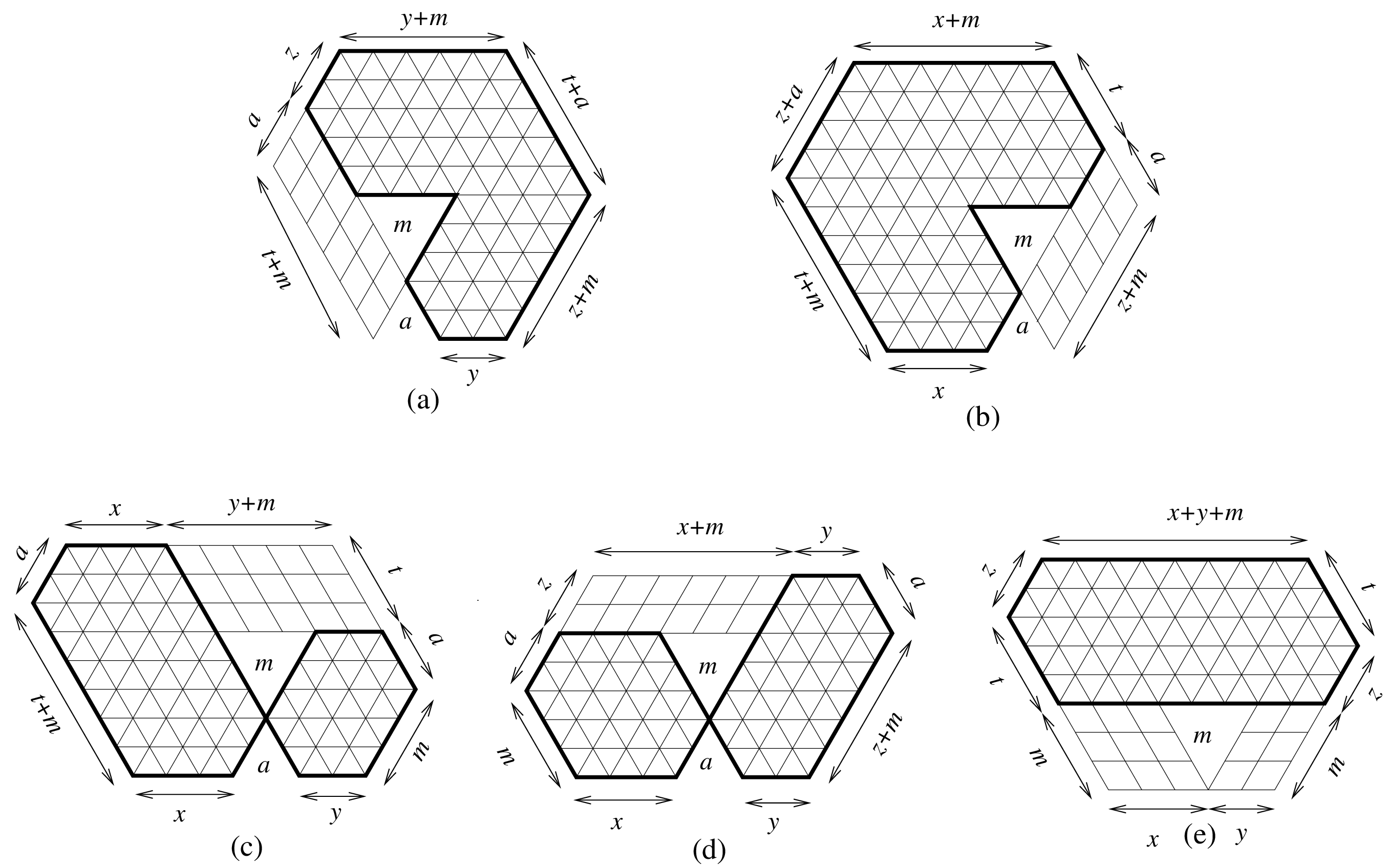}
\caption{The base cases in the proofs of Propositions \ref{hexprop4} and \ref{hexprop5}: (a) $x=0$, (b) $y=0$, (c) $z=0$, (d) $t=0$, and (e) $a=0$.}
\label{basecase2}
\end{figure}

If $z=0$, by applying Region-splitting Lemma \ref{GS} as in Figure \ref{basecase2}(c), we get
\begin{equation}
\M_2\big(B_{m,a}(x,y,0,t)\big)=\M_2\big(Hex(m,y,a)\big)\M_2\Big(B_{m,a}(x,y,0,t)-Hex(m,y,a)\Big).
\end{equation}
Next, we remove the forced left lozenges  from the region $B_{m,a}(x,y,0,t)-Hex(m,y,a)$  and obtain the hexagon $Hex(a,x,t+m)$ weighted by $wt_2$. Then we get
\begin{equation}
\M_2\big(B_{m,a}(x,y,0,t)\big)=\M_2\big(Hex(m,y,a)\big)\M_2\big(Hex(a,x,t+m)\big),
\end{equation}
and (\ref{qmagnet}) follows from  Corollary \ref{hexprop1}.

If $t=0$, similar to the case when $z=0$, the Region-splitting Lemma \ref{GS} implies
\begin{equation}
\M_2\big(B_{m,a}(x,y,z,0)\big)=\M_2\big(Hex(z+m,y,a)\big)\M_2\Big(B_{m,a}(x,y,z,0)-Hex(z+m,y,a)\Big)
\end{equation}
(see Figure \ref{basecase2}(d)).
We also get the hexagon $Hex(a,x,m)$ (weighted by $wt_2$) after removing forced lozenges from the region $B_{m,a}(x,y,z,0)-Hex(z+m,y,a)$. However, our forced lozenges are now right lozenges, which have weight product equal to $q^{(m+a)(x+m)z+(x+m)\binom{z+1}{2}}$. Thus, we get
\[\M_2\Big(B_{m,a}(x,y,z,0)-Hex(z+m,y,a)\Big)=q^{(m+a)(x+m)z+(x+m)\binom{z+1}{2}}\M_2\big(Hex(a,x,m)\big),\] so
\begin{equation}
\M_2\big(B_{m,a}(x,y,z,0)\big)=q^{(m+a)(x+m)z+(x+m)\binom{z+1}{2}}\M_2\big(Hex(z+m,y,a)\big)\M_2\big(Hex(a,x,m)\big).
\end{equation}
Again, (\ref{qmagnet}) is implied by Corollary \ref{hexprop1}.

\medskip

For the induction step, we assume that $m,a, y,z,t\geq 1$ and that (\ref{qmagnet}) holds for any magnet bar regions, which have the sum of the $y$-, $z$- and $t$-parameters strictly less than $y+z+t$.

 \begin{figure}\centering
\includegraphics[width=12cm]{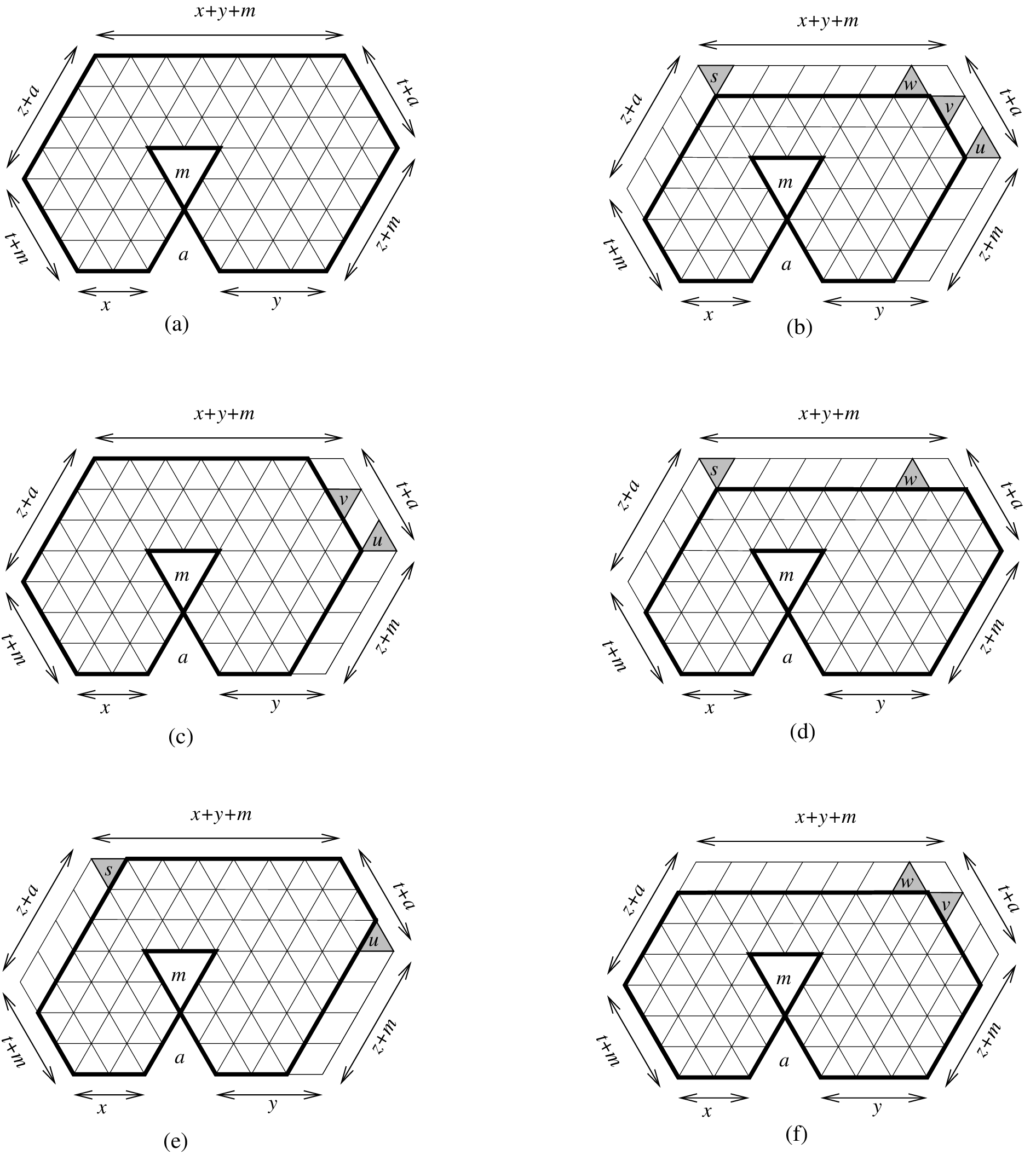}
\caption{Obtaining the recurrence for the tiling generating functions of  magnet bar regions.}
\label{kuomagnet}
\end{figure}
We apply Kuo Theorem \ref{kuothm} to the dual graph $G$ of the magnet bar region $B_{m,a}(x,y,z,t)$ (weighted by $wt_2$). We pick the four vertices $u, v, w,s$ as in Figure \ref{kuomagnet}(b). In particular, the four shaded unit triangles correspond to the four vertices: the shaded unit triangle corresponding to $u$ is the lowest one, and  $v, w,s$ correspond to the next shaded unit triangles as we move counter-clockwise from the lowest one. We notice that the north side of the region has length $x+y+m\geq y+m\geq 2$ and the northeast side has length $t+a\geq 2$, so the four vertices $u,v,w,s$ are well-defined.

By removing the lozenges  forced by the shaded unit triangles, we get back new $B$-type regions weighted by $wt_2$.
Collecting the weights of those forced  lozenges, we get
\begin{equation}\label{maineq5b}
\M(G-\{u,v,w,s\})=q^{\binom{z+m+1}{2}+(x+y+m-2)(z+t+m+a)}\M_2\big(B_{m,a}(x,y-1,z,t-1)\big),
\end{equation}
\begin{equation}\label{maineq1b}
\M(G-\{u,v\})=q^{\binom{z+m+1}{2}}\M_2\big(B_{m,a}(x,y-1,z,t)\big),
\end{equation}
\begin{equation}\label{maineq2b}
\M(G-\{w,s\})=q^{(x+y+m-2)(z+t+m+a)}\M_2\big(B_{m,a}(x,y,z,t-1)\big),
\end{equation}
\begin{equation}\label{maineq3b}
\M(G-\{u,s\})=q^{\binom{z+m+1}{2}}\M_2\big(B_{m,a}(x,y-1,z+1,t-1)\big) ,
\end{equation}
and
\begin{equation}\label{maineq4b}
\M(G-\{v,w\})=q^{(x+y+m-1)(z+t+m+a)}\M_2\big(B_{m,a}(x,y,z-1,t)\big)
\end{equation}
(see Figures \ref{kuomagnet}(b)--(f), respectively).
Plugging the above identities into the equation (\ref{kuoeq}) in  Kuo Condensation Theorem \ref{kuothm},  we obtain
\begin{align}
\M_2\big(B_{m,a}(x,y,z,t)\big)&\M_2\big(B_{m,a}(x,y-1,z,t-1)\big)=\notag\\
&\M_2\big(B_{m,a}(x,y-1,z,t)\big)\M_2\big(B_{m,a}(x,y,z,t-1)\big)\notag\\
&+q^{z+t+m+a}\M_2\big(B_{m,a}(x,y-1,z+1,t-1)\big)\M_2\big(B_{m,a}(x,y,z-1,t)\big).
\end{align}
All regions in the above equation, except for the first one, have the sum of their $y$-, $z$- and $t$-parameters strictly less than $y+z+t$. Thus, by the induction hypothesis, those regions have their
tiling generating functions given by (\ref{qmagnet}). By substituting these formulas into the above equation and performing some simplifications, one readily gets $\M_2\left(B_{m,a}(x,y,z,t)\right)$ equal exactly to the expression on the right-hand side of (\ref{qmagnet}). This finishes our proof.
\end{proof}

We need another tiling $q$-enumeration of the magnet bar  region as follows.

\medskip

Assume that we now give all right and left lozenges in the magnet bar region $B_{m,a}(x,y,z,t)$ a weight 1. Next, we give a \emph{vertical} lozenge a weight $q^{l}$, where $l$ is the distance between the \emph{northeast} side of the lozenge and the \emph{southwest} side of the region. We denote by $wt_3$ the new weight assignment, and $\M_3$ the corresponding tiling generating function.

\begin{prop}\label{hexprop5} For non-negative integers $m,a,x,y,z,t$
   \begin{align}\label{qmagnet2}
 \M_3&\big(B_{m,a}(x,y,z,t)\big)
=q^{m\binom{a+1}{2}+t\binom{z+a+1}{2}+a(z+m)(x+a)+a\binom{z+m+1}{2}}\notag\\
&\times \frac{\Hf_q(m+a+x+y+z+t)}{\Hf_q(m+a+x+y+t)\Hf_q(m+a+x+y+z)}\notag\\
 &\times\frac{\Hf_q(m+a+x+t)\Hf_q(m+a+x+y)\Hf_q(m+a+y+z)}{\Hf_q(m+a+z+t)\Hf_q(m+a+x)
 \Hf_q(m+a+y)}\notag\\
 &\times \frac{\Hf_q(x)\Hf_q(y)\Hf_q(z)\Hf_q(t)\Hf_q(m)\Hf_q(a)^2}{\Hf_q(a+x)\Hf_q(a+y)\Hf_q(z+t)\Hf_q(m+a)}\notag\\
 &\times\frac{\Hf_q(m+z+t)\Hf_q(m+a+x)\Hf_q(m+a+y)}{\Hf_q(m+y+z)\Hf_q(m+x+t)}.
 \end{align}
\end{prop}
\begin{proof}
The equality (\ref{qmagnet2}) can be treated similarly to (\ref{qmagnet}) in the Proposition \ref{hexprop4} by induction on $y+z+t$.

We would also like to obtain a recurrence for the $\M_3$-generating function of the magnet bar by using Kuo consdensation. Similar to Proposition \ref{hexprop4}, the application of Kuo condensation requires that the north and the northeast sides both have lengths greater than $1$.  Taking into account also the base cases of the recurrence, we need to verify (\ref{qmagnet2}) in the situations when $a=0$, $x=0$, $y=0$, $z=0$ or $t=0$.

Assume that our magnet bar region is weighted by $wt_3$. We still follow the process in Figures \ref{basecase2} and \ref{kuomagnet}, however, the reader should be aware that the weight assignment here is \emph{different} from that in the proof of Proposition \ref{hexprop4}.

If $a=0$, we get a weighted version of the hexagon $Hex(z,x+y+m,t)$ by removing forced lozenges from the magnet bar region as in Figure \ref{basecase2}(e). Rotating the hexagon  $60^{\circ}$ clockwise and reflecting the resulting region over a vertical line, we get the hexagon $Hex(z,t,x+y+m)$ weighted by $wt_2$, and (\ref{qmagnet2}) follows from Corollary \ref{hexprop1}.

If $x=0$, after removing forced vertical lozenges (whose weight product is equal to $q^{(t+m)\binom{a+1}{2}}$) as in Figure \ref{basecase2}(a), we get a weighted region $R$. Next, we rotate $R$ $60^{\circ}$ counter-clockwise and reflect the resulting region about a vertical line. This way, we get the weighted region $K_{m}(a,t,z,y)$ in which a right lozenge is weighted by $q^{a+l}$, where $l$ is the distance from the top of the lozenge to the bottom of the region. We divide the weight of each right lozenge in the latter region by $q^{a}$, and get back the weight assignment $wt_2$. Thus,
\begin{equation}
\M_3\big(B_{m,a}(0,y,z,t)\big)=q^{(t+m)\binom{a+1}{2}}q^{azt+a^2(z+m)}\M_2\big(K_{m}(a,t,z,y)\big),
\end{equation}
where the factor $q^{azt+a^2(z+m)}$ comes from the weight division. Then (\ref{qmagnet2}) follows from Lemma \ref{hexprop3b}.

If $y=0$, by removing forced lozenges (whose weight product is $q^{(x+a)(z+m)a+a\binom{z+m+1}{2}}$) and rotating the resulting region $60^{\circ}$ clockwise, we get a weighted version of region $K_{m}(a,z,t,x)$ in which a right lozenge is weighted by $q^{m+a+x+z+1-l}$, where $l$ is the distance from the left side of the lozenge and the southeast side of the region (see Figure \ref{basecase2}(b)). By dividing the weight of each right lozenge by $q^{m+a+x+z+1}$, we get back the weight assignment $wt_1$, \emph{where $q$ is replaced by $q^{-1}$}. Thus, (\ref{qmagnet2}) follows from Proposition \ref{hexprop2}, Lemma \ref{hexprop3b} and the simple fact $[n]_{q^{-1}}=[n]_q/q^{n-1}$.

 If $z=0$, we apply Region-splitting Lemma \ref{GS} (and remove forced lozenges weighted by $1$) to split our region into two hexagons as in Figure \ref{basecase2}(c). Next, we rotate the right hexagon $60^{\circ}$ counter-clockwise  and reflect the resulting hexagon about a vertical line to get the hexagon $Hex(a,t+m,x)$ weighted by $wt_2$. For the left hexagon, we also rotate it $60^{\circ}$ counter-clockwise, reflect the resulting region about a vertical line and divide the weight of each right lozenge of it by $q^{x+a}$ to get the hexagon $Hex(m,a,y)$ weighted by $wt_2$. Then we get (\ref{qmagnet2}) from Corollary \ref{hexprop1}. The case when $t=0$, can be treated similarly to the case when $z=0$, based on Figure \ref{basecase2}(d).

\medskip

The induction step is completely analogous to that of the proof of Proposition \ref{qmagnet}. We also apply Kuo's Theorem \ref{kuothm}, based on Figure \ref{kuomagnet}. The vertices $u,v,w,s$ are well-defined because the north side of the magnet bar has length $x+y+m\geq x+y\geq 2$ and the northeast side has length $t+a\geq 2$.  After removing lozenges forced by the shaded unit triangles, we get back new $B$-type regions weighted by $wt_3$.  Figure \ref{kuomagnet} tells us that the product of $\M_3$-generating functions of the two regions on the top is equal to the product of the $\M_3$-generating functions of the two regions in the middle plus the product of $\M_3$-generating functions of the two regions on the bottom. To be precise, we get the following recurrence
\begin{align}
\M_3\big(B_{m,a}(x,y,z,t)&\big)\M_3\big(B_{m,a}(x,y-1,z,t-1)\big)=\notag\\
&\M_3\big(B_{m,a}(x,y-1,z,t)\big)\M_3\big(B_{m,a}(x,y,z,t-1)\big)\notag\\
&+q^{m+a+x+y+z}\M_3\big(B_{m,a}(x,y-1,z+1,t-1)\big)\M_3\big(B_{m,a}(x,y,z-1,t)\big),
\end{align}
and the proposition follows from the induction hypothesis.
\end{proof}

\begin{rmk} In some sense, Proposition \ref{hexprop5} is equivalent to Proposition \ref{hexprop4} in the same way as the three weightings in Proposition \ref{hexprop2} are equivalent.
Indeed, we first introduce an analog  $wt'$ of the natural weight assignment $wt_0$, by viewing each vertical lozenge in a tiling $T$ as the right face of a horizontal block running from left to right in the pile $\pi_T$ corresponding to $T$. Each vertical lozenge is now weighted by $q^x$, where $x$ is the length of its corresponding block. Using the same arguments as used when comparing $wt_2$ and $wt_0$ in Proposition \ref{hexprop2}, one can show that $wt_3$ and $wt'$ are only different by some power of $q$. Moreover, we note that $wt'(T)=wt_0(T)=q^{|\pi_T|}$, for any tiling $T$ of the region. This implies that $wt_2$ and $wt_3$ are only different by some multiplicative factor, and Proposition \ref{hexprop5} follows from Proposition \ref{hexprop4} (and vice versa).
\end{rmk}

\section{Proof of Theorem \ref{qmain}}

By Proposition \ref{hexprop2}, we only need to show that
  \begin{align}\label{qmaineq}
 \M_2&\left(Q\begin{pmatrix}x&y&z&t\\m&a&b&c\end{pmatrix}\right)=\notag\\
&q^{\textbf{g}\begin{pmatrix}x&y&z&t\\m&a&b&c\end{pmatrix}} \frac{\Hf_q(m+a+b+c+x+y+z+t)}{\Hf_q(m+a+b+c+x+y+t)\Hf_q(m+a+b+c+x+y+z)}\notag\\
 &\times\frac{\Hf_q(m+a+b+c+x+t)\Hf_q(m+a+b+c+x+y)\Hf_q(m+a+b+c+y+z)}{\Hf_q(m+a+b+c+z+t)\Hf_q(m+a+b+c+x)
 \Hf_q(m+a+b+c+y)}\notag\\
 &\times \frac{\Hf_q(x)\Hf_q(y)\Hf_q(z)\Hf_q(t)}{\Hf_q(x+t)\Hf_q(y+z)}\frac{\Hf_q(m)^3\Hf_q(a)^2\Hf_q(b)\Hf_q(c)\Hf_q(m+a+b+c)}
 {\Hf_q(m+a)^2\Hf_q(m+b)\Hf_q(m+c)}\notag\\
 &\times\frac{\Hf_q(m+b+c+z+t)\Hf_q(m+a+c+x)\Hf_q(m+a+b+y)}{\Hf_q(m+b+y+z)\Hf_q(m+c+x+t)}\notag\\
  &\times\frac{\Hf_q(c+x+t)\Hf_q(b+y+z)}{\Hf_q(a+c+x)\Hf_q(a+b+y)\Hf_q(b+c+z+t)}.
 \end{align}

\begin{figure}\centering
\includegraphics[width=14cm]{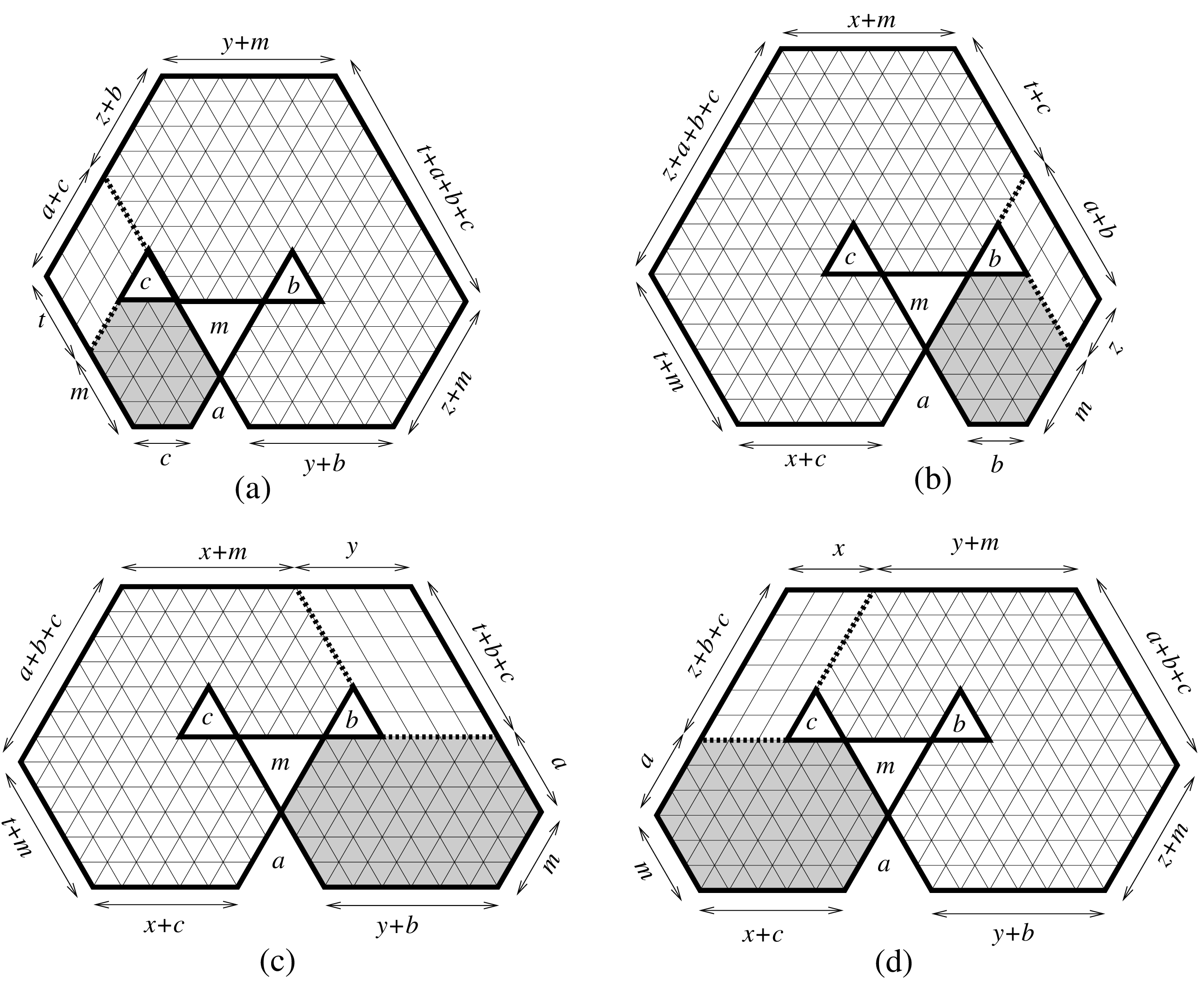}
\caption{The four base cases: (a) $x=0$, (b) $y=0$, (c) $z=0$ (b), and $t=0$.}
\label{basecase}
\end{figure}

\normalsize We prove (\ref{qmaineq}) by induction on $y+z+t$.  We assume that our region $Q:=Q\begin{pmatrix}x&y&z&t\\m&a&b&c\end{pmatrix}$ is weighted by $wt_2$

Similarly to the case of the magnet bars in Propositions \ref{hexprop4} and \ref{hexprop5}, we will use Kuo's condensation Theorem \ref{kuothm} to obtain a recurrence for the $\M_2$-generating function of the $Q$-type region. In order to apply the theorem, certain side-lengths of the region $Q$ must be large enough. In particular, this requires $x+y+m\geq 2$ and $t+a+b+c\geq 2$. Taking into account also the base cases of the recurrence, we need to verify  (\ref{qmaineq}) in the cases when one of the four parameters $x,y,z,t$ equals $0$.

If $x=0$, by applying Region-splitting Lemma \ref{GS},  we split $Q$ into two parts as in Figure \ref{basecase}(a): the shaded hexagon $Hex(a,c,m)$ (weighted by $wt_2$) and $Q - Hex(a,c,m)$. After removing forced vertical lozenges (which have the weight 1) from the latter region, we get a weighted magnet bar region (rotated $60^{\circ}$). Rotating the magnet bar region $60^{\circ}$ counter-clockwise and reflecting the resulting region about a vertical line, we get the magnet bar region $B_{b,m}(a,t+c,z,y)$ weighted by $wt_3$.
Thus, we have
\begin{align}
\M_2\left(Q\begin{pmatrix}0&y&z&t\\m&a&b&c\end{pmatrix}\right)&=\M_2\big(Hex(a,c,m)\big)\M_3\left(Q\begin{pmatrix}0&y&z&t\\m&a&b&c\end{pmatrix}-Hex(a,c,m)\right)\notag\\
&=\M_2\big(Hex(a,c,m)\big)\M_3\big(B_{b,m}(a,t+c,z,y)\big),
\end{align}
and (\ref{qmaineq}) follows from Corollary \ref{hexprop1} and Proposition \ref{hexprop5}. The case $t=0$ can be treated similarly to the case $x=0$, based on Figure \ref{basecase}(d). The only difference is that our forced right lozenges have weight product equal to $q^{x(a+m)(z+b+c)+x\binom{z+b+c+1}{2}}$. Thus, we get
\begin{align}
\M_2&\left(Q\begin{pmatrix}x&y&z&0\\m&a&b&c\end{pmatrix}\right)\notag\\
&=\M_2\big(Hex(a,x+c,m)\big)\M_2\left(Q\begin{pmatrix}x&y&z&0\\m&a&b&c\end{pmatrix}-Hex(a,x+c,m)\right)\notag\\
&=\M_2\big(Hex(a,x+c,m)\big)q^{x(a+m)(z+b+c)+x\binom{z+b+c+1}{2}} \M_3\big(B_{b,m}(a,c,z,y)\big).
\end{align}
Again, (\ref{qmaineq}) follows from Corollary \ref{hexprop1} and Proposition \ref{hexprop5}.

If $y=0$,  by Region-splitting Lemma \ref{GS}, we get
\begin{equation}
\M_2\left(Q\begin{pmatrix}x&0&z&t\\m&a&b&c\end{pmatrix}\right)=\M_2\big(Hex(m,b,a)\big)\M_2\left(Q\begin{pmatrix}x&0&z&t\\m&a&b&c\end{pmatrix}-Hex(m,b,a)\right)
\end{equation}
(see Figure \ref{basecase}(b)).
We also remove forced lozenges (having weight 1) from the second region on the right-hand side  to get a region $R'$. Next, we rotate $R'$ $60^{\circ}$ clockwise and reflect it about a vertical line to get the magnet bar $B_{c,m}(b+z,a,x,t)$ weighted by $wt_1$. Thus, we have
\begin{equation}
\M_2\left(Q\begin{pmatrix}x&0&z&t\\m&a&b&c\end{pmatrix}\right)=\M_2\big(Hex(m,b,a)\big)\M_1\big(B_{c,m}(b+z,a,x,t)\big),
\end{equation}
and (\ref{qmaineq}) follows from Corollary \ref{hexprop1} and Propositions \ref{hexprop2} and \ref{hexprop4}. The case $z=0$ can be obtained in the same way, based on Figure \ref{basecase}(c).

\begin{figure}\centering
\includegraphics[width=14cm]{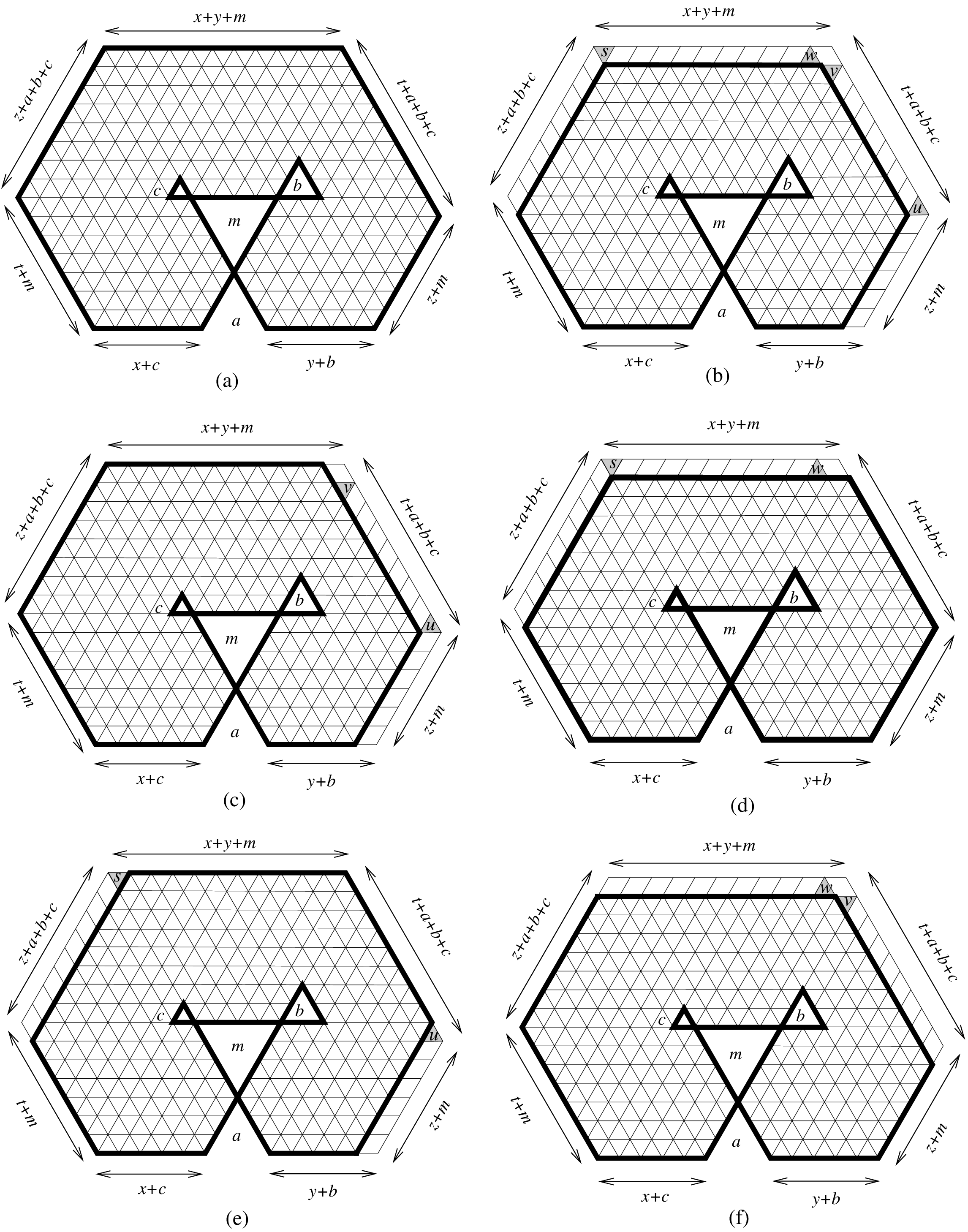}
\caption{Obtaining the recurrence on the numbers of tilings by using Kuo's condensation.}
\label{kuoshamrock}
\end{figure}

\medskip

For the induction step, we assume that $x,y,z,t$ are positive and that (\ref{qmaineq}) holds for any  $Q$-type regions in which the sum of the $y$-, $z$- and $t$-parameters is strictly less than $y+z+t$.

If $b=c=0$, then (\ref{qmaineq}) follows directly from Proposition \ref{hexprop4}. Therefore, we can assume, from now on, that $b+c\geq1$.

We now apply Kuo condensation to the dual graph $G$ of the region $Q\begin{pmatrix}x&y&z&t\\m&a&b&c\end{pmatrix}$ weighted by $wt_2$. The four vertices $u,v,w,s$ correspond to the four shaded unit triangles in Figure \ref{kuoshamrock}(b). We notice that the north side of the region has length $x+y+m\geq x+y\geq 2$ and the northeast side has length $t+a+b+c\geq t+b+c\geq 2$, so the four vertices $u,,v,w,s$ are well-defined. By collecting the weights of the forced lozenges shown in Figures \ref{kuoshamrock}(b)--(f), we get respectively
\begin{equation}\label{qmaineq5}
\M(G-\{u,v,w,s\})=q^{\binom{z+m+1}{2}+(x+y+m-2)(z+t+m+a+b+c)}\M_2\left(Q\begin{pmatrix}x&y-1&z&t-1\\m&a&b&c\end{pmatrix}\right),
\end{equation}
\begin{equation}\label{qmaineq1}
\M(G-\{u,v\})=q^{\binom{z+m+1}{2}}\M_2\left(Q\begin{pmatrix}x&y-1&z&t\\m&a&b&c\end{pmatrix}\right),
\end{equation}
\begin{equation}\label{qmaineq2}
\M(G-\{w,s\})=q^{(x+y+m-2)(z+t+m+a+b+c)}\M_2\left(Q\begin{pmatrix}x&y&z&t-1\\m&a&b&c\end{pmatrix}\right),
\end{equation}
\begin{equation}\label{qmaineq3}
\M(G-\{u,s\})=q^{\binom{z+m+1}{2}}\M_2\left(Q\begin{pmatrix}x&y-1&z+1&t-1\\m&a&b&c\end{pmatrix}\right) ,
\end{equation}
and
\begin{equation}\label{qmaineq4}
\M(G-\{v,w\})=q^{(x+y+m-1)(z+t+m+a+b+c)}\M_2\left(Q\begin{pmatrix}x&y&z-1&t\\m&a&b&c\end{pmatrix}\right).
\end{equation}
Substituting (\ref{qmaineq5})--(\ref{qmaineq4}) into equation (\ref{kuoeq}) in Kuo's Theorem \ref{kuothm}, we get
\small{\begin{align}\label{qmaineq6}
\M_2&\left(Q\begin{pmatrix}x&y&z&t\\m&a&b&c\end{pmatrix}\right)\M_2\left(Q\begin{pmatrix}x&y-1&z&t-1\\m&a&b&c\end{pmatrix}\right)\notag\\
&=\M_2\left(Q\begin{pmatrix}x&y-1&z&t\\m&a&b&c\end{pmatrix}\right)\M_2\left(Q\begin{pmatrix}x&y&z&t-1\\m&a&b&c\end{pmatrix}\right)\notag\\
&+q^{z+t+m+a+b+c}\M_2\left(Q\begin{pmatrix}x&y-1&z+1&t-1\\m&a&b&c\end{pmatrix}\right)\M_2\left(Q\begin{pmatrix}x&y&z-1&t\\m&a&b&c\end{pmatrix}\right).
\end{align}}

\normalsize Finally, if we denote by $\Psi\begin{pmatrix}x&y&z&t\\m&a&b&c\end{pmatrix}$ the expression on the right-hand side of (\ref{qmaineq}), we only need to show that $\Psi$ also satisfies the recurrence (\ref{qmaineq6}). Equivalently, we need to verify that
 \small{\begin{align}\label{qcheckeq1}
 \frac{\Psi\begin{pmatrix}x&y-1&z&t\\m&a&b&c\end{pmatrix}}{\Psi\begin{pmatrix}x&y&z&t\\m&a&b&c\end{pmatrix}}
 &\frac{\Psi\begin{pmatrix}x&y&z&t-1\\m&a&b&c\end{pmatrix}}{\Psi\begin{pmatrix}x&y-1&z&t-1\\m&a&b&c\end{pmatrix}}+\notag\\
 &q^{z+t+m+a+b+c}\frac{\Psi\begin{pmatrix}x&y&z-1&t\\m&a&b&c\end{pmatrix}}{\Psi\begin{pmatrix}x&y&z&t\\m&a&b&c\end{pmatrix}}
 \frac{\Psi\begin{pmatrix}x&y-1&z+1&t-1\\m&a&b&c\end{pmatrix}}{\Psi\begin{pmatrix}x&y-1&z&t-1\\m&a&b&c\end{pmatrix}}=1.
 \end{align}}
\normalsize Let $\Phi\begin{pmatrix}x&y&z&t\\m&a&b&c\end{pmatrix}:=q^{-\textbf{g}\begin{pmatrix}x&y&z&t\\m&a&b&c\end{pmatrix}}\Psi\begin{pmatrix}x&y&z&t\\m&a&b&c\end{pmatrix}$. We notice that the function $\Phi\begin{pmatrix}x&y&z&t\\m&a&b&c\end{pmatrix}$ is simply the expression on the right-hand side of (\ref{mainequation}). By the definition of the function $\textbf{g}$, we get
\small{\begin{equation}
\textbf{g}\begin{pmatrix}x&y-1&z&t\\m&a&b&c\end{pmatrix}+\textbf{g}\begin{pmatrix}x&y&z&t-1\\m&a&b&c\end{pmatrix}=\textbf{g}\begin{pmatrix}x&y&z&t\\m&a&b&c\end{pmatrix}+\textbf{g}\begin{pmatrix}x&y-1&z&t-1\\m&a&b&c\end{pmatrix}
\end{equation}}
\normalsize and
\small{\begin{align}
\textbf{g}\begin{pmatrix}x&y&z-1&t\\m&a&b&c\end{pmatrix}+&\textbf{g}\begin{pmatrix}x&y-1&z+1&t-1\\m&a&b&c\end{pmatrix}\notag\\&=(x+y-z-1)+\textbf{g}\begin{pmatrix}x&y&z&t\\m&a&b&c\end{pmatrix}+\textbf{g}\begin{pmatrix}x&y-1&z&t-1\\m&a&b&c\end{pmatrix}.
\end{align}}
\normalsize Therefore,  (\ref{qcheckeq1}) is equivalent to
 \small{\begin{align}\label{qcheckeq2}
 \frac{\Phi\begin{pmatrix}x&y-1&z&t\\m&a&b&c\end{pmatrix}}{\Phi\begin{pmatrix}x&y&z&t\\m&a&b&c\end{pmatrix}}
& \frac{\Phi\begin{pmatrix}x&y&z&t-1\\m&a&b&c\end{pmatrix}}{\Phi\begin{pmatrix}x&y-1&z&t-1\\m&a&b&c\end{pmatrix}}+\notag\\
& q^{m+a+b+c+x+y+t-1}\frac{\Phi\begin{pmatrix}x&y&z-1&t\\m&a&b&c\end{pmatrix}}{\Phi\begin{pmatrix}x&y&z&t\\m&a&b&c\end{pmatrix}}
 \frac{\Phi\begin{pmatrix}x&y-1&z+1&t-1\\m&a&b&c\end{pmatrix}}{\Phi\begin{pmatrix}x&y-1&z&t-1\\m&a&b&c\end{pmatrix}}=1.
 \end{align}}

\normalsize Let us simplify the first term on the left-hand side of (\ref{qcheckeq2}). We notice that the two $\Phi$-functions in the numerator and denominator of the first fraction in the first term are different only at  their $y$-parameters. Canceling out all terms, which have no $y$-parameter, and using the trivial fact $\Hf_q(n+1)/\Hf_q(n)=[n]_q!$,  we get
\small{\begin{align}
&\frac{\Phi\begin{pmatrix}x&y-1&z&t\\m&a&b&c\end{pmatrix}}{\Phi\begin{pmatrix}x&y&z&t\\m&a&b&c\end{pmatrix}}=\frac{[y+z-1]_q![a+b+y-1]_q![m+b+y+z-1]_q!}{[y-1]_q![b+y+z-1]_q![m+a+b+y-1]_q!}\notag\\
&\times \frac{[m+a+b+c+y-1]_q![m+a+b+c+x+y+t-1]_q![m+a+b+c+x+y+z-1]_q!}{[m+a+b+c+x+y-1]_q![m+a+b+c+y+z-1]_q![m+a+b+c+x+y+z+t-1]_q!}.
\end{align}}
\normalsize Doing similarly for the second fraction of the first term, we obtain
\small{\begin{align}
& \frac{\Phi\begin{pmatrix}x&y&z&t-1\\m&a&b&c\end{pmatrix}}{\Phi\begin{pmatrix}x&y-1&z&t-1\\m&a&b&c\end{pmatrix}}=\frac{[y-1]_q![b+y+z-1]_q![m+a+b+y-1]_q!}{[y+z-1]_q![a+b+y-1]_q![m+b+y+z-1]_q!}\notag\\
&\times \frac{[m+a+b+c+x+y-1]_q![m+a+b+c+y+z-1]_q![m+a+b+c+x+y+z+t-2]_q!}{[m+a+b+c+y-1]_q![m+a+b+c+x+y+t-2]_q![m+a+b+c+x+y+z-1]_q!}.
\end{align}}
\normalsize Thus, the first term on the left-hand side of (\ref{qcheckeq2}) can be simplified as
\small{\begin{equation}\label{qcheckeq3}
 \frac{\Phi\begin{pmatrix}x&y-1&z&t\\m&a&b&c\end{pmatrix}}{\Phi\begin{pmatrix}x&y&z&t\\m&a&b&c\end{pmatrix}}
 \frac{\Phi\begin{pmatrix}x&y&z&t-1\\m&a&b&c\end{pmatrix}}{\Phi\begin{pmatrix}x&y-1&z&t-1\\m&a&b&c\end{pmatrix}}=\frac{[m+a+b+c+x+y+t-1]_q}{[m+a+b+c+x+y+z+t-1]_q}.
\end{equation}}
\normalsize We simplify the second term on the left-hand side of (\ref{qcheckeq2}) in the same way (the numerator and the denominator in each fraction are now different at their $z$-parameters). We get
\small{\begin{equation}\label{qcheckeq4}
 \frac{\Phi\begin{pmatrix}x&y&z-1&t\\m&a&b&c\end{pmatrix}}{\Phi\begin{pmatrix}x&y&z&t\\m&a&b&c\end{pmatrix}}
 \frac{\Phi\begin{pmatrix}x&y-1&z+1&t-1\\m&a&b&c\end{pmatrix}}{\Phi\begin{pmatrix}x&y-1&z&t-1\\m&a&b&c\end{pmatrix}}=\frac{[z]_q}{[m+a+b+c+x+y+z+t-1]_q}.
\end{equation}}
\normalsize By (\ref{qcheckeq3}) and (\ref{qcheckeq4}), the equality (\ref{qcheckeq2}) becomes the following identity
\small{\begin{equation}
\frac{[m+a+b+c+x+y+t-1]_q+q^{m+a+b+c+x+y+t-1}[z]_q}{[m+a+b+c+x+y+z+t-1]_q}=1,
\end{equation}}
\normalsize which follows directly from the definition of the $q$-integers. This completes our proof.

\subsection*{Acknowledgements}
I would like to thank the two anonymous reviewers for their helpful suggestions and comments. The current proof of Lemma \ref{GS} was provided by one of the reviewers.

I also thank  Mihai Ciucu for fruitful discussions, and David Wilson for showing me the $q$-mode of his  software \texttt{vacmax 1.6e} (available for downloading at David Wilson's website, \texttt{http://dbwilson.com/vaxmacs/}), which is really helpful in verifying the formulas in the paper.

 This research was supported in part by the Institute for Mathematics and its Applications with funds provided by the National Science Foundation (grant no. DMS-0931945).

%\bigskip
%\textbf{References}

\end{document}